\documentclass[12pt,a4paper]{amsart}

\usepackage{amsmath, amssymb}
\usepackage{amscd}
\usepackage{amsthm}

\usepackage[normalem]{ulem}
\usepackage{mathabx}

\usepackage{multirow}
\usepackage{changepage, array}
\usepackage{xcolor}
\usepackage{graphicx}
\usepackage{comment}
\usepackage{hyperref}
\usepackage{cleveref}
\usepackage{enumerate}
\usepackage{tikz,ifthen} 
\usetikzlibrary{arrows.meta}
\usetikzlibrary{bending}
\usepackage{pgfplots}
\usetikzlibrary{positioning}
\pgfplotsset{compat=1.8}

\newcommand{\frakg}{{\mathfrak{g}}}

\newcommand{\frakk}{{\mathfrak{k}}}
\newcommand{\frakl}{{\mathfrak{l}}}
\newcommand{\frakm}{{\mathfrak{m}}}

\newcommand{\frakq}{{\mathfrak{q}}}

\newcommand{\fraku}{{\mathfrak{u}}}

\newcommand{\res}{{\mathrm{Res}}}
\newcommand{\Sym}{{\mathrm{Sym}}}

\newcommand{\rA}{{\mathrm{A}}}
\newcommand{\rB}{{\mathrm{B}}}
\newcommand{\rC}{{\mathrm{C}}}
\newcommand{\rD}{{\mathrm{D}}}
\newcommand{\rE}{{\mathrm{E}}}
\newcommand{\rF}{{\mathrm{F}}}
\newcommand{\rG}{{\mathrm{G}}}
\newcommand{\rU}{{\mathrm{U}}}
\newcommand{\tU}{\widetilde{\mathrm{U}}}

\newcommand{\rSp}{{\mathrm{Sp}}}

\newcommand{\SU}{{\mathrm{SU}}}
\newcommand{\SL}{{\mathrm{SL}}}

\newcommand{\Spin}{{\mathrm{Spin}}}
\newcommand{\Vmin}{{\mathrm{V}}_{\mathrm{min}}}

\newcommand{\calF}{{\mathcal{F}}}
\newcommand{\calJ}{{\mathcal{J}}}
\newcommand{\calK}{{\mathcal{K}}}
\newcommand{\calO}{{\mathcal{O}}}
\newcommand{\calU}{{\mathcal{U}}}
\newcommand{\calV}{{\mathcal{V}}}
\newcommand{\bfA}{{\mathbf{A}}}
\newcommand{\ZZ}{{\bbZ / 2 \bbZ}}
\newcommand{\spin}{\widetilde{Spin}(4,4)}
\newcommand{\spf}{\widetilde{Spin}(5,4)}

\newcommand{\PU}{\textnormal{PU}}
\newcommand{\talpha}{\tilde{\alpha}}
\newcommand{\Hom}{{\mathrm{Hom}}}

\newcommand{\bbC}{{\mathbb{C}}}
\newcommand{\bbP}{{\mathbb{P}}}

\newcommand{\bbZ}{{\mathbb{Z}}}

\newtheorem{lemma}{Lemma}[section]

\newtheorem{prop}[lemma]{Proposition}
\newtheorem{thm}[lemma]{Theorem}
\newtheorem{cor}[lemma]{Corollary}
\newtheorem{claim}{Claim}

\theoremstyle{remark}

\newtheorem{rmk}[lemma]{Remark}

\newcommand{\red}{\color{red}}

\usepackage{parskip}

\newcommand{\SubS}{\subsection{}}

\dedicatory{
to Dick Gross, in memoriam} 

\title[Exceptional dual pair correspondences; real groups of split rank one]{Exceptional dual pair correspondences; case of real groups of split rank one}
\author{Petar Baki\'c} 
\address{Department of Mathematics, University of Utah, Salt Lake City, UT 84112}
 \email{bakic@utah.edu}

\author{Hung Yean Loke}
\address{Department of Mathematics, 
National University of Singapore,
Block S17, 10 Lower Kent Ridge Road, Singapore 119076.}
\email{matlhy@nus.edu.sg}

\author{Gordan Savin}

\address{Department of Mathematics, University of Utah, Salt Lake City, UT 84112}
 \email{gordan.savin@utah.edu}

\begin{document}


\footnote {This work is  supported by the Croatian Science Foundation under the project IP-2022-10-4615, a gift No. 946504 from the Simons Foundation and an NUS MOE grant A-8002493-00-00}

\begin{abstract}
Exceptional real groups have quaternionic forms of split rank 4 that contain dual pairs $G\times G'$, 
 where $G'$ is the split Lie group of the type $\rG_2$, and $G$ a Lie group of split rank one. 
  In this paper we restrict the minimal representation of the quaternionic  group 
 to the dual pair and prove some significant results for the resulting correspondence of representations. 
\end{abstract}

\maketitle

\section{Introduction}  

We start with a general situation.  Let $\mathfrak g$ and $\mathfrak g'$ be a pair of complex simple Lie algebras. Let $G$ and $G'$ be a pair of Lie groups with 
complexified Lie algebras $\mathfrak g$ and $\mathfrak g'$. Let $K$ and $K'$ be the maximal compact subgroups of $G$ and $G'$, respectively. In this paper, 
we shall work with $(\mathfrak g,K)$-modules, and study the following problem:

Let $(\omega, \Omega)$ be a $(\mathfrak g \times \mathfrak g', K\times K')$-module, where $\omega$ denotes the action on the vector space $\Omega$. Let 
$\pi$ be an irreducible $(\mathfrak g,K)$-module. Then there exists a $(\mathfrak g',K')$-module $\Theta(\pi)$, the big theta lift of $\pi$,  such that 
\[ 
\Omega /\bigcap_{T\in \Hom_{\mathfrak g}(\Omega,\pi)} \mathrm{ker} (T) \cong \pi \otimes \Theta(\pi). 
\] 
Conversely, starting with an irreducible $(\mathfrak g',K')$-module $\pi'$, we can define a $(\mathfrak g,K)$-module $\Theta(\pi')$.  In order to 
get a correspondence of irreducible modules, one would like to show that~$\Theta(\pi)$ has finite length and that it has a unique irreducible quotient, and 
the same for~$\Theta(\pi')$. 

 We state a general result in that direction. To do so, we need some additional notation.  Let $Z(\mathfrak g)$ and $Z(\mathfrak g')$ be the centers 
 of enveloping algebras.  If $\tau$ is a $K$-type (that is an irreducible finite-dimensional representation of $K$), let $\Theta(\tau)$ denote the lift of $\tau$, i.e.\ the $\mathfrak g'$-module such that the $\tau$-isotypic summand of $\Omega$ is $\Omega[\tau]=\tau \otimes \Theta(\tau)$.  Then we have (see Theorem \ref{T:general}): 
\begin{thm} \label{T:introduction} 
 Assume that the following two hold: 
\begin{itemize}  
\item There is a correspondence of infinitesimal characters, that is, $\omega(Z(\mathfrak g))= \omega(Z(\mathfrak g'))$.  
\item  For every $K$-type $\tau$, there exists a finite dimensional representation $F_{\tau}$ of $K'$ such that
$\Theta(\tau)$ is a quotient of $U(\mathfrak g')\otimes_{ U(\mathfrak k')} F_{\tau}$.  
\end{itemize} 
Let $\pi$ and $\pi'$  be irreducible $(\mathfrak g,K)$ and  $(\mathfrak g',K')$-modules, respectively. Then: 
\begin{itemize}  
\item $\Theta(\pi)$ and $\Theta(\pi')$ are finite length $(\mathfrak g',K')$ and  $(\mathfrak g,K)$-modules, respectively. 
\item If $\tau$ is a $K$-type, then 
\[ 
\dim_{K}\Hom (\Theta(\pi'), \tau) \leq \dim\Hom_{K'}(F_{\tau}, \pi'). 
\] 
\end{itemize} 
\end{thm} 
 In this paper we look at the case where $\Omega$ is the minimal representation of the quaternionic group $\rE_{n,4}$ constructed by Gross and Wallach \cite{GW1} and \cite{GW2}.  
The exceptional groups $\rF_{4,4}$, $\rE_{n,4}$, $n=6,7,8$,  contain  two families of dual pairs $G \times G'$ where $G'$ is the split exceptional group of type $\rG_2$, 
and 
\[ 
G =\mathrm{Aut}(J) 
\] 
where $J$ is a Freudenthal-Jordan algebra \cite{KMRT}. As a vector space, $J$ is the space of $3\times 3$ hermitian-symmetric matrices with coefficients in 
$\mathbb R$, $\mathbb C$, $\mathbb H$ and $\mathbb O$, where the latter is the algebra of Cayley octonions. The Jordan algebra structure depends on the 
choice of the identity $e$. If 
\[ 
e= 
\left(\begin{array}{ccc} 
1 & & \\
& 1 &  \\
& & 1
\end{array} 
\right)
\] 
then $G$ is compact. The simply connected cover of $G$ is, respectively, $\Spin(3)$, $\SU(3)$, $\rSp(3)$ and $\rF_4$.
 In this case the restriction of $\Omega$ is a direct sum of $\pi\otimes \Theta(\pi)$. Irreducibility and a complete 
description of $\Theta(\pi)$ was obtained in \cite{HPS}.
In this paper we consider the other family, for 
\[ 
e= 
\left(\begin{array}{ccc} 
1 & & \\
& -1 &  \\
& & -1
\end{array} 
\right) 
\] 
when $G$ is the split rank one form of the compact group in the first family. 
 As explained in~\cite{LiDuke},  the correspondence obtained in \cite{HPS} for the first family implies the first assumption of Theorem \ref{T:introduction} for the second family. The crux of this paper is the proof of the second assumption (existence of $F_\tau$) for groups in the second family, that is, the rank one $G$.  
 To that end, let $K$ be the maximal compact subgroup of $G$.  
 The centralizer of $K$ is a split simply connected group of
type $B_3$, thus we have the following see-saw.  
\begin{center}
	\begin{tikzpicture}[>=stealth, scale=0.85]
		\node (ThetaPi)  at (-1.5,0) {$\Theta(\pi)$};
		\node (ThetaTau) at (-1.5,2.5) {$\Theta(\tau)$};
		
		\draw[->] (-1.5,1.8) -- (-1.5,0.7);
		
		\node (G2) at (0,0) {$G_2$};
		\node (B3) at (0,2.5) {$B_3$};
		
		\node (K) at (3,0) {$K$};
		\node (tau) at (4,0) {$\tau$};
		
		\draw[<-] (4,1.8) -- (4,0.7);
		
		\node (G)  at (3,2.5) {$G$};
		\node (pi) at (4,2.5) {$\pi$};
		
		\draw (G2) -- (G);
		\draw (B3) -- (K);
	\end{tikzpicture}
\end{center}
In the above picture $\tau$ is a $K$-type of $\pi$. Observe that $\Theta(\tau)$ is naturally a $B_3$-module. It is of interest to us since, as the picture shows, 
$\Theta(\pi)$ is a quotient of $\Theta(\tau)$. 
 We prove that $\Theta(\tau)$ is in fact an irreducible quaternionic representation of $B_3$, with an explicit minimal type $F_{\tau}$.  
We also prove a general result showing that any quaternionic representation of $B_3$ is generated by its minimal type when restricted to $G_2$. 
Thus, if $\mathfrak g_2$ is the complexified Lie algebra of $G_2$, it follows at once that $\Theta(\tau)$ is a quotient of 
$U(\mathfrak g_2) \otimes_{U(\mathfrak k_2)} F_{\tau}$, as desired. 

Thus, Theorem \ref{T:introduction} holds for our dual pair. The first conclusion, finite length, is of obvious importance. 
The second goes a long way towards determining $\Theta(\pi)$ from $\pi'$, and works well for cohomological $\pi'$, 
since $K'$-types of such representations lie in explicit cones, see \cite{VoganZuckerman} and \cite{SR}.
Indeed, we obtain some very precise results for the dual pair $\mathrm{PU}(2,1)\times G_2$ in $E_{6,4}$, see Theorem \ref{G2toPU3regular}.  In particular, we prove Conjecture 4.5 in \cite{BHLS} 
(functoriality of the correspondence) with the assumption of unitarizability (see Theorem \ref{T:main_wall}), 
 and this is enough for applications in loc.\ cit.\  and \cite{BHLS_2} since local components of square integrable automorphic representations are unitary.  

The paper is organized as follows. In Section \ref{S:general} we revisit the definition of $\Theta(\pi)$ and prove Theorem  
\ref{T:introduction} using results proved in subsequent sections. In Section \ref{S:QG} we briefly review quaternionic groups. In Section \ref{S:QR} we introduce quaternionic representations and 
explicate the Lie algebra action on them. This section contains some key results, such as Theorem \ref{T:generation},  that is used in Section \ref{S:B3G2} where we prove that 
quaternionic representations of $B_3$, when restricted to $G_2$, are generated by the minimal type.  Sections \ref{S:D4} to \ref{S:F4} are used to compute the 
lift $\Theta(\tau)$ from $K$ to $B_3$. The method uses another see-saw, involving a split group of type $D_4$; however, computations have to be done on a case by case basis, since exceptional groups do not have structural uniformity as, for example, general linear or symplectic groups. In Section \ref{S:unitary} we review the representations of $\mathrm{PU}(2,1)$ and $G_2$ that play a role in Section \ref{S:correspondence}, where we compute the theta lifts from $G_2$ to $\PU(2,1)$ for cohomological representations. Finally, in Section \ref{S:branching}, we gather some branching rules and, using the $B_3$ correspondence in $E_8$ from Section \ref{S:E8},  derive a branching rule from $F_4$ to $B_4$, for a two-parameter family of finite dimensional representations of $F_4$, in the style of  \cite{HTW}.  

\noindent {\bf Acknowledgement.} Hung Yean Loke would like to thank the hospitality of University of Utah where part of this paper was written. Petar Baki\'{c} would like to thank Aleksander Horawa, Siyan Daniel Li-Huerta, and Naomi Sweeting for the many useful conversations about the $\mathrm{PU}(3)\times G_2$ correspondence.



\section{Main result on dual pair correspondences}  \label{S:general}

In this section we derive some general results on theta correspondences and prove the result announced in the introduction. To that end, we introduce a  more flexible definition of 
$\Theta(\pi)$ that does not require that $\pi$ is irreducible.  

\subsection{Two definitions of $\Theta(\pi)$} 
Let $(\omega, \Omega)$ be a $(\mathfrak g \times \mathfrak g', K\times K')$-module, where $\omega$ denotes the action on the vector space $\Omega$.  The usual definition of 
$\Theta(\pi)$ requires the following:  

\begin{prop}  \label{P:big_theta_def} 
	Let $\pi$ be an irreducible $(\mathfrak g,K)$-module. 
	Then there exists a $(\mathfrak g',K')$-module $\Theta(\pi)$, the big theta lift of $\pi$,  such that 
	\[ 
	\Omega /\bigcap_{\varphi \in \Hom_{\mathfrak g}(\Omega,\pi)} \mathrm{ker} (\varphi ) \cong \pi \otimes \Theta(\pi). 
	\] 
\end{prop} 
The proposition clearly follows from the following special case: 
\begin{lemma}  Assume that $\cap_{\varphi \in \Hom_{\mathfrak g}(\Omega,\pi)} \mathrm{ker} (\varphi )=0$. Then, we have a 
	natural isomorphism 
	\[ 
	\Omega\cong \Hom_{\mathfrak g}(\pi,\Omega)\otimes \pi. 
	\] 
\end{lemma} 
\begin{proof} The key is to show that every $v\in\Omega$ is contained in a finite sum of copies of $\pi$. Since $v$ is $K$-finite, there exists 
	a finite dimensional subpace $\pi(v)\subseteq \pi$ such that $\varphi(v)\in \pi(v)$ for all $\varphi\in  \Hom_{\mathfrak g}(\Omega,\pi)$.  
	Since $\pi(v)$ is finite-dimensional, there 
	exists $\varphi_1, \ldots ,\varphi_n$ such that every $\varphi(v)$ is a linear combination of $\varphi_1(v), \ldots ,\varphi_n(v)$ . Let 
	$V=U(\mathfrak g)\cdot v \subseteq  \Omega$. It follows that 
	\[ 
	w\mapsto (\varphi_1(w), \ldots ,\varphi_n(w)), \, w\in V
	\] 
	is an embedding of $V$ into $\pi^n$.  Since $\pi$ is irreducible, it is easy to see that $V\cong \pi^k$ for some $k\leq n$, and hence $v$ is in the image of the 
	natural map 
	\[ 
	\Hom_{\mathfrak g}(\pi,\Omega) \otimes \pi \rightarrow \Omega
	\] 
	$(f,v)\mapsto f(v)$. 
\end{proof}  
We shall now give another description of $\Theta(\pi)$. For any $(\mathfrak g,K)$-module $V$, let $V_{\mathfrak g}$ denote the maximal quotient of $V$ 
such that $\mathfrak g$ acts trivially on it.  
\begin{prop}  Let $\pi$ be an irreducible $(\mathfrak g,K)$-module. Then 
	\[ 
	\Theta(\pi)\cong (\Omega\otimes \pi^{\vee})_{\mathfrak g}, 
	\] 
	where $\pi^{\vee}$ denotes the contragredient of $\pi$. 
\end{prop}
\begin{proof} 
	Let $\Theta(\pi)^*$ be the linear dual of $\Theta(\pi)$. 
	Observe that we have a natural isomorphism 
	\[ 
	\Theta(\pi)^* \cong  \Hom_{\mathfrak g}(\Omega,\pi). 
	\] 
	Next, we have a surjection. 
	\[ 
	(\Omega\otimes \pi^{\vee})_{\mathfrak g} \rightarrow (\Theta(\pi)\otimes \pi\otimes\pi^{\vee})_{\mathfrak g} \cong \Theta(\pi). 
	\] 
	Take the linear duals of both sides, we get an injection in the opposite direction, 
	\[ 
	\Theta(\pi)^*    \cong  \Hom_{\mathfrak g}(\Omega,\pi)
	\rightarrow  [(\Omega\otimes \pi^{\vee})_{\mathfrak g}]^{\ast}. 
	\] 
	Now observe that $[(\Omega\otimes \pi^{\vee})_{\mathfrak g}]^{\ast}$ is the same as $\mathfrak g$-invariant bilinear forms on $\Omega \times \pi^{\vee}$. Hence 
	\[ 
	[(\Omega\otimes \pi^{\vee})_{\mathfrak g}]^{\ast} \cong \Hom_{\mathfrak g}(\Omega, (\pi^{\vee})^{\vee}). 
	\] 
	Since $\pi$ is irreducible (and therefore admissible), $ (\pi^{\vee})^{\vee}\cong\pi$. Hence the injection is a bijection. 
\end{proof} 
In view of the above result, we can define $\Theta(\pi)$ for any  $(\mathfrak g, K)$-module $\pi$ by 
\[ 
\Theta(\pi)=(\Omega\otimes \pi^{\vee})_{\mathfrak g}. 
\] 
This definition has several advantages. 
The module $\pi$ does not have to be irreducible. It is clear that $\Theta(\pi)$ is a $(\mathfrak g', K')$-module.  Most importantly, it is easier to work with.  

\subsection{Main result} If $\tau$ is a $K$-type,  we can define $\Theta(\tau)=(\Omega\otimes\tau^{\vee})_K$, the lift of $\tau$, in the same way. 
It is a $(\mathfrak g',K')$-module. 
The following was announced in the introduction. We now present the proof.  
\begin{thm} \label{T:general} 
	Assume that the following two hold: 
	\begin{itemize}  
		\item There is a correspondence of infinitesimal characters, that is, $\omega(Z(\mathfrak g))= \omega(Z(\mathfrak g'))$.  
		\item  For every $K$-type $\tau$, there exists a finite dimensional representation $F_{\tau}$ of $K'$ such that
		$\Theta(\tau)$ is a quotient of $U(\mathfrak g')\otimes_{ U(\mathfrak k')} F_{\tau}$.  
	\end{itemize} 
	Let $\pi$ and $\pi'$  be irreducible $(\mathfrak g,K)$ and  $(\mathfrak g',K')$-modules, respectively. 
	Then
	\begin{itemize}  
		\item  $\Theta(\pi)$ and $\Theta(\pi')$ are finite length $(\mathfrak g',K')$ and  $(\mathfrak g,K)$-modules, respectively. 
		\item If $\tau$ is a $K$-type, then 
		\[ 
		\dim \Hom_K(\Theta(\pi'), \tau) \leq \dim \Hom_{K'} (F_{\tau}, \pi'). 
		\] 
	\end{itemize}
\end{thm} 
\begin{proof}  Let $\tau$ be a $K$-type. Since $\pi\otimes \Theta(\pi)$ is a quotient of $\Omega$, we have a surjection 
	\[ 
	\Theta(\tau) = (\Omega\otimes\tau^{\vee})_K  \rightarrow (\pi \otimes \tau^{\vee})_K \otimes \Theta(\pi).  
	\] 
	Hence, if $\tau$ is a type of $\pi$, it follows that $\Theta(\pi)$ is a quotient of $\Theta(\tau)$.  Hence 
	$\Theta(\pi)$ is a quotient of $U(\mathfrak g')\otimes_{ U(\mathfrak k')} F_{\tau}$, a finitely generated module. 
	Furthermore, by the first assumption,  $Z(\mathfrak g')$ acts on $\Theta(\pi)$ by the infinitesimal character, corresponding to the 
	infinitesimal character of $\pi$. 
	Finitely generated plus infinitesimal character implies finite length. Hence $\Theta(\pi)$ has finite length. 
	
	
	Note that we do not assume that the second bullet holds with the roles of the two algebras switched. Hence we need a different argument to prove that $\Theta(\pi')$ has finite length:  the inequality, $\dim \Hom_K(\Theta(\pi'), \tau) \leq \dim \Hom_{K'} (F_{\tau}, \pi')$. To prove it, notice that we have the see-saw identity (switching the order of taking $\mathfrak g'$ and $K$-coinvariants) 
	\[ 
	\Hom_K(\Theta(\pi'), \tau) \cong \Hom_{\mathfrak g'} (\Theta(\tau), \pi'). 
	\] 
	Since $\Theta(\tau)$ is a quotient of $U(\mathfrak g')\otimes_{ U(\mathfrak k')} F_{\tau}$, the inequality follows from Frobenius reciprocity. 
	From the inequality we see that $\Theta(\pi')$ is admissible. This and infinitesimal character implies finite length. 
\end{proof} 

\subsection{Our dual pair} Recall that we are interested in dual pairs $G\times G'$ in quaternionic groups where $G'$ is the split Lie group of type $G_2$ 
and $G=\mathrm{Aut}(J)$ where $J$ is the Jordan algebra of $3\times 3$ hermitian matrices with coefficients $\mathbb R$, $\mathbb C$, $\mathbb H$ and $\mathbb O$, 
and the identity 
\[ 
\left(\begin{array}{ccc} 
	1 & & \\
	& -1 &  \\
	& & -1
\end{array} 
\right). 
\] 
The module $\Omega$ is the minimal representation $\Vmin$.  Let $K$ be the maximal compact subgroup of $G$. We have a see-saw diagram 
\begin{equation} \label{eqseesaw_general}
	\arraycolsep=1.5pt\def\arraystretch{1.35}
	\begin{array}{ccc}
		G & & \Spin(4,3)\\
		| & \bigtimes & | \\
		K & &  G'=\rG_{2,2} 
	\end{array}
\end{equation} 
Let $\tau$  be a $K$-type. In Sections \ref{S:E6}--\ref{S:F4} we prove that $\Theta(\tau)$ is an irreducible quaternionic representation of $\Spin(4,3)$. Let $F_{\tau}$ be its minimal type.  By Theorem  \ref{T:restriction_g2}, $\Theta(\tau)$ is $\mathfrak g'$-generated by $F_{\tau}$. In particular, the conditions of Theorem \ref{T:general} are satisfied 
for the dual pair $G\times G'$. Hence: 
\begin{cor}  The conclusions of Theorem \ref{T:general} hold for the dual pairs $G\times G'$ and $\Omega=\Vmin$, the minimal representation of the ambient quaternionic exceptional group. 
\end{cor} 

\section{Quaternionic groups} \label{S:QG}

Let $\mathfrak g$ be a simple complex Lie algebra. In this paper we shall study representations of the quaternionic 
real group with the complexified Lie algebra isomorphic to $\mathfrak g$. Since we  work in the Language of $(\mathfrak g, K)$-modules, our first task is 
to describe the corresponding Cartan involution and the decomposition $\mathfrak g=\mathfrak k \oplus \mathfrak p$. 

\subsection{General case}  \label{S11} Fix a maximal Cartan algebra $\mathfrak t$. Let $\Phi$ be the corresponding root system. Pick a system 
$\Delta =\{\alpha_1, \ldots , \alpha_{\ell}\}$ of simple roots. Let $\alpha_0$ be the lowest root.
Let $G_{\mathbb C}$ be the corresponding Chevalley group (of adjoint type).  Let 
$\varphi_0: \SL(2,\mathbb C) \rightarrow G_{\mathbb C}$ be a homomorphism corresponding to $\alpha_0$. 
Then the Cartan involution is  
\[ 
\theta= \varphi_0 
\left( 
\begin{array}{cc} 
-1 & \\ 
& -1 
\end{array} 
\right). 
\]   
We shall now give a more detailed description of  the corresponding Cartan decomposition. 
Let $(e,h,f)$ be an $\mathfrak{sl}(2)$-triple corresponding to  the highest root $-\alpha_0$.   
Then the  centralizer of~$h$ in $\mathfrak g$ is a standard Levi subgroup $\mathfrak l$, corresponding to the set of simple roots perpendicular to $\alpha_0$.  
The nilpotent radical $\mathfrak n$ of the standard 
parabolic subgroup $\mathfrak q= \mathfrak l \oplus \mathfrak n$ has a decomposition $\mathfrak n_1 \oplus \mathfrak n_2$ given by 
$h$-grading. Then $\mathfrak n$ is a Heisenberg Lie algebra with the center $\mathfrak n_2 = \mathbb C \cdot e$. The nilpotent radical 
of the opposite parabolic $\bar{\mathfrak q}$ is 
 $\bar{\mathfrak n} = \mathfrak n_{-1}  \oplus \mathfrak n_{-2}$ where $\mathfrak n_{-2} = \mathbb C \cdot f$.  Let $\mathfrak m = [\mathfrak l, \mathfrak l]$. 
 If $\Phi$ is not type $\rA_{\ell}$ then $\mathfrak l = \mathfrak m \oplus \mathbb C \cdot h$.  
 The two summands in the Cartan decomposition are 
 \[ 
\mathfrak k= \mathfrak{sl}(2) \oplus \mathfrak m \text { and }  \mathfrak p= \mathfrak n_{-1}\oplus \mathfrak n_1. 
\] 
Let $K$ and $M$ be  the simply connected compact Lie groups with the complexified Lie algebra isomorphic to $\mathfrak k$ and 
 $\mathfrak m$, respectively. Then 
\[ 
K=\SU_0(2) \times M. 
\] 
We denote the first factor by $\SU_0(2)$, in order to distinguish it from other groups isomorphic to $\SU(2)$.  
Since $[f,\mathfrak n_1] =\mathfrak n_1$, and $M$ commutes with $f$, we see that  $\mathfrak n_{-1}$ and $\mathfrak n_1$ are isomorphic as $M$-modules. Let us denote this representation by $V_M$. Then $\mathfrak p \cong (1) \otimes V_M$, as $\SU_0(2) \times M$-modules, where $(m)$, throughout the text, denotes 
the irreducible representation of $\SU(2)$ with the highest weight $m$.  We list some cases in the following table.

\begin{table}[h]
\begin{center}
		\caption{The group $M$ and its representation $V_M$}
\begin{tabular}{l|l|l}
$G$ & $M$ & $V_M$  \\
\hline
 \hbox{\vbox{\vskip 5pt \hbox{$\Spin(d,4)$}}} & $\SU(2) \times \Spin(d)$ & $\bbC^2 \otimes \bbC^d$  \\ 
$\rE_{6,4}$ & $\SU(6) $ & $\wedge^3\mathbb C^6$ \\
$\rE_{7,4}$ & $\Spin(12)$ & $\mathbb C^{32}$  \\
$\rE_{8,4}$ & $\rE_7$ & $\mathbb C^{56}$   \\
$\rF_{4,4}$ & $\rSp(3)$ & $\wedge^3\mathbb C^6/\mathbb C^6$
\end{tabular}
\end{center}
\label{table0}
\end{table}

Here $\mathbb C^{32}$ and $\mathbb C^{56}$ are the spin and the miniscule 56-dimensional representations of $\Spin(12)$ and $\rE_7$, respectively. 
We note that the adjoint forms of exceptional quaternionic groups are topologically connected, see \cite[Section 10.3]{AT}. 
Hence their irreducible representations are the same as irreducible representations of the simply connected cover with the trivial central character.

\subsection{Example of $\rE_6$} \label{SrootsysmE6}
Let $\Phi(\rE_6)$ denote the root system of type $\rE_6$.
We label the root system as according to \cite{Bou}: 
The set of positive roots consists of 
\begin{align*}
& \pm \varepsilon_i + \varepsilon_j \ \ \text{ for } 1 \leq i < j \leq 5, \\
& \dfrac{1}{2}(\varepsilon_8 - \varepsilon_7 - \varepsilon_6 \pm \varepsilon_1 \pm \varepsilon_2 \pm \cdots \pm \varepsilon_5) \\ & {} \ \text{ (even number of negative signs)}.
\end{align*}
The simple roots are
\begin{align}
& \alpha_1 =  \dfrac{1}{2}(\varepsilon_1 - \varepsilon_2 - \varepsilon_3 - \varepsilon_4 - \varepsilon_5 + \varepsilon_8 - \varepsilon_7 - \varepsilon_6), \nonumber \\
& \alpha_2 = \varepsilon_1 + \varepsilon_2, \ \ \alpha_3 = \varepsilon_2 - \varepsilon_1, \ \ \alpha_4 = \varepsilon_3 -\varepsilon_2, \ \ \alpha_5 = \varepsilon_4 - \varepsilon_3, \ \ \alpha_6 = \varepsilon_5 - \varepsilon_4. \label{eqsimpleroots}
\end{align}
The extended root system is 
 
\begin{picture}(200,140)(-120,-25)

\put(79,23){\line(0,-1){10}}
\put(79,-5){\circle{6}}

\put(79,73){\line(0,-1){30}}
\put(79,40){\circle{6}}


\put(74,29){$\alpha_2$}

\put(02,82){$\alpha_1$}

\put(38,82){$\alpha_3$}

\put(74,82){$\alpha_4$}

\put(110,82){$\alpha_5$}

\put(146,82){$\alpha_6$}

\put(74, 2){$\alpha_0$}

\put(07,76){\circle{6}}
\put(10,76){\line(1,0){30}}

\put(43,76){\circle{6}}
\put(46,76){\line(1,0){30}}
\put(79,76){\circle{6}}

\put(82,76){\line(1,0){30}}
\put(115,76){\circle{6}}

\put(118,76){\line(1,0){30}}
\put(151,76){\circle{6}}

\end{picture}
where $-\alpha_0$ is the highest root 
\[ 
-\alpha_0 =  \dfrac{1}{2}(\varepsilon_1 + \varepsilon_2 + \varepsilon_3 + \varepsilon_4 + \varepsilon_5 +\varepsilon_8 - \varepsilon_7 - \varepsilon_6)
= \alpha_1 + 2 \alpha_2 + 2 \alpha_3 + 3 \alpha_4 + 2 \alpha_5 + \alpha_6.
\] 

Now it is evident that $K\cong \SU_0(2)\times \SU(6)$, corresponding to the diagram with $\alpha_2$ removed, 
and the highest weight of $\mathfrak p$ is $-\alpha_2$. Hence, as a $K$-module, $\mathfrak p$ is a tensor product 
of the standard two-dimensional representation of $\SU_0(2)$, and the third fundamental representation of $\SU(6)$, whose highest weight is (represented by) 
$(1,1,1,0,0,0)$.


\section{Quaternionic representations} \label{S:QR} 

In this section we give a self contained treatment of quaternionic representations of groups $G$ introduced in the previous section: their construction, $K$-types and Lie algebra action. 
We continue with the notation introduced in the previous section. 

\subsection{Construction} Let $W_M$ be a finite dimensional representation of $\mathfrak m$ (this is the same as a representation of $M$). 
Let $s\geq 2$ be an integer.  
Extend this to a representation $W_M[s]$ of $\mathfrak l=\mathfrak m \oplus \mathbb C\cdot h$ so that $h$ acts by multiplication by $s$. 
Consider the generalized Verma module 
\[ 
V= V(\mathfrak g, W_M[s])= U(\mathfrak g)\otimes_{U(\bar{\mathfrak q})}W_M[s] \cong U(\mathfrak n)\otimes W_M[s]. 
\] 
If $W_M$ is irreducible with the highest weight $\mu$, then the infinitesimal character of $V$ is 
\begin{equation} \label{E:inf_A_general} 
\mu + s\alpha_0/2 +\rho
\end{equation} 
 where $\rho$ is the half sum of positive roots. 
 Let $L = \rU(1) \times  M \subset K$ where the complexified Lie algebra of $\rU(1)$ is $\mathbb C\cdot h$.  Observe that $V$ is a $(\mathfrak g,L)$-module. 
Let $K_1=\SU(2)$ (the first factor of $K$) and $L_1=L\cap  K_1=\rU(1)$. Let $\Gamma_{K_1/L_1}$ be the corresponding Zuckerman functor. 
Let 
\[ 
\bfA=\bfA(\mathfrak g, W_M[s]) =  \Gamma^1_{K_1/L_1}(V). 
\] 
It is a $(\mathfrak g, K)$-module.  
In this section we shall study $\bfA$, in  particular, we shall prove  that $\bfA$ is generated by its minimal $K$-type $(s-2)\otimes W_M$. 

\subsection{Figuring out $K$-types} 
Since the Zuckerman functor in degree 0 amounts to taking $\mathfrak k$-finite vectors, our first task is to understand $V$ as a $\mathfrak k$-module. To that end, 
recall the $\mathfrak{sl}(2)$-triple, $(e,h,f)$ where $e\in \mathfrak n_2$,   $f\in \mathfrak n_{-2}$, and $h=[e,f]$, spanning the first summand of $\mathfrak k$. 
Let $U(\mathfrak n)$ be the enveloping algebra  of $\mathfrak n$. It  is  a free module over $U(\mathfrak n_2)\cong \mathbb C[e]$ 
(here $e$ is of the $\mathfrak{sl}(2)$  and there is an obvious filtration 
of $U(\mathfrak n)$ by $U(\mathfrak n_2)$-submodules $U_k(\mathfrak n)$ such that 
\[ 
U_0(\mathfrak n)= U(\mathfrak n_2) \text{ and } U_k(\mathfrak n)/U_{k-1}(\mathfrak n)\cong U(\mathfrak n_2) \otimes S^k(V_M).  
\] 
The filtration of $U(\mathfrak n)$ gives a filtration $V_k = U(\mathfrak n)_k\otimes W_M[s]$ of $V$. 
An easy check shows that the filtration is invariant under the action of $\mathfrak{sl}(2)=\langle e,h,f\rangle$, and 
\[ 
 V_k/V_{k-1}\cong V(s+k) \otimes S^k(V_M)\otimes W_M 
\]  
as $\mathfrak{sl}(2) \times M$-module, where $V(s+k)$ is irreducible $\mathfrak{sl}(2)$ module of lowest $h$-weight equal to $s+k$.  
Since $s\geq 2$, these have different infinitesimal characters. Hence the filtration splits, 
\[ 
V=\bigoplus_{k\geq 0} V(s+k) \otimes \tau_k 
\] 
where $\tau_k\cong S^k(V_M)\otimes W_M$.  Since $\Gamma_{K_1/L_1}^1(V(s+k))=(s+k-2)$, 
 the irreducible representation of $\SU(2)$ with the highest weight $s+k-2$, it follows that
\begin{equation} \label{E:K_types} 
\bfA=\bfA(G, W_M[s]) = \Gamma^1(V)= \bigoplus_{k=0}^{\infty} (s+k-2) \otimes \tau_k. 
\end{equation} 
This gives us a $K$-types description of $\bfA$. Observe that $\bfA(G, W_M[s])$ is $\SU_0(2)$-admissible.  

\subsection{Lie algebra action} In order to compute the action of $\mathfrak p$ on $\bfA$, we first need to do that for $V$. 
The action of $\mathfrak p=(1)\otimes V_M$  on $V$ is a $K$-equivariant homomorphism 
\begin{equation} \label{E:action_V} 
\pi : \mathfrak p\otimes V \rightarrow V. 
\end{equation} 
Let $\pi_k$ be the restriction of $\pi$ to $\mathfrak p \otimes V(s+k)\otimes \tau_k$. Clearly $\pi$ is the sum of $\pi_k$.  
Since the tensor product of $(1)$, the factor of $\mathfrak p$, and the lowest weight module $V(k+s)$ decomposes as 
\[ 
(1)\otimes V(s+k) \cong V(s+k-1) \oplus V(s+k+1), 
\] 
 we can decompose 
\[
\mathfrak p\otimes (V(s+k) \otimes \tau_k) \cong V(s+k-1)\otimes (V_M \otimes \tau_k) \,\, \oplus  \,\, V(s+k+1)  \otimes (V_M \otimes \tau_k). 
\]
Accordingly, we can also decompose $\pi_k=\pi_k^-\oplus \pi_k^+$, a sum of restrictions of $\pi$ to the two summands above. 
The images of $\pi_k^-$ and $ \pi_k^+$ sit, 
respectively, in $V(s+k-1)\otimes \tau_{k-1}$ and $V(s+k+1)\otimes \tau_{k+1}$. Summarizing, $\pi$ is a sum of the maps 
\begin{equation}\label{two_in_V}
\pi^{\pm}_k:  V(s+k \pm 1) \otimes V_M \otimes  \tau_k \rightarrow V(s+k \pm 1 ) \otimes \tau_{k \pm 1}. 
\end{equation}
Since $V(s+k \pm 1)$ are irreducible $\mathfrak{sl}(2)$-modules, we can write $\pi^{\pm}_k =1\otimes \sigma_k^{\pm}$ for a couple of  $M$-homomorphisms
\[ 
\sigma_{k}^-:  V_M \otimes \tau_k \rightarrow \tau_{k-1} \text { and } \sigma_{k}^+:  V_M \otimes \tau_k \rightarrow \tau_{k+1}. 
\] 
Now our next task is to explicate $\sigma_k^{\pm}$. 
To that end, observe that $\tau_k$ can be defined as the subspace of $V$ as follows: 
\[ 
\tau_k=\{ v\in V ~|~ f\cdot v=0 \text{ and } h\cdot v = (s+k) v \}. 
\] 

\begin{lemma}  Let $v\in \tau_k$, $x\in \mathfrak n_1$ and $y=[f,x]\in \mathfrak n_{-1}$. Then 
\begin{itemize} 
\item $y\cdot v\in \tau_{k-1}$. 
\item $x\cdot v\in \tau_{k+1} \oplus e\cdot \tau_{k-1}$. 
\end{itemize} 
\end{lemma} 
\begin{proof} Since $f$ and $y$ commute, $f\cdot (y\cdot v)= y\cdot (f\cdot v)=0$.  Moreover, 
$y\cdot v$ has $h$-weight $s+k-1$. Hence $y\cdot v\in \tau_{k-1}$.  For the second bullet, the $h$-weight of $x\cdot v$ is $s+k+1$.  Hence 
\[ 
x\cdot v \in \tau_{k+1} + e\cdot \tau_{k-1} + \cdots.  
\] 
Since $[f,x]=y$ and $f\cdot v=0$,  it follows that $f\cdot (x\cdot v)= y \cdot v\in \tau_{k-1}$. Hence 
\[ 
x\cdot v = v_1 + \frac{1}{s+k-1} y \cdot v 
\] 
with $v_1\in  \tau_{k+1}$.  If $v$ corresponds to  $p\otimes w\in S^k(\mathfrak n_1) \otimes W_M$, then $v_1$ corresponds to  $xp\otimes w$ 
(natural multiplication of $x$ and $p$).  The formula for $y\cdot v$ should involve some differentiation of $p$ by $y$ (perhaps it is better to call it
differentiation by $x$).
\end{proof} 

Now, from the (proof of the) lemma it is easy to see that 
 \[ 
 \sigma_k^{+} : V_M  \otimes S^k(V_M)\otimes W_M \rightarrow  S^{k+1}(V_M)\otimes W_M
 \] 
 is given by $(x\otimes p\otimes w) \mapsto  xp \otimes W$, in particular, $\sigma_{k}^+$ is \underline{surjective}. 

Now we can derive the action of $\mathfrak p$ on $\bfA$, as follows.  The action of $\mathfrak p$ on $V$ is given by a $K$-equivariant map (\ref{E:action_V}). 
By functoriality, we have 
\begin{equation} \label{E:action_A} 
\Gamma^1(\mathfrak \pi) : \Gamma^1(\mathfrak p \otimes V) \rightarrow \Gamma^1(V). 
\end{equation} 
Since $\mathfrak p$ is $K$-finite, we have a natural isomorphism $\Gamma^1(\mathfrak p \otimes V)\cong \mathfrak p\otimes  \Gamma^1(V)$ \cite[page 177]{RRG1},
and thus the above map can be reinterpreted as the action of $\mathfrak p$ on $\Gamma^1(V)$ \cite[page 179]{RRG1}. Now recall that $\pi$ is the sum of  $\pi_k^{\pm}$ in (\ref{two_in_V}). 
Hence, by functoriality, $\Gamma^1(\pi)$ is a sum of $K$-maps
\[ 
\Gamma^1(\pi^{-}_k) :  (s+k -3) \otimes V_M \otimes  \tau_k \rightarrow (s+k -3 ) \otimes \tau_{k - 1} 
\] 
and 
\[ 
\Gamma^1(\pi^{+}_k) :  (s+k -1) \otimes V_M \otimes  \tau_k \rightarrow (s+k -1 ) \otimes \tau_{k + 1}. 
\] 
Recall that $\pi_k^{\pm}= 1\otimes \sigma_k^{\pm}$ where $1$ is the identity on $V(s+k\pm 1)$ and 
$\sigma_k^{\pm} : V_M \otimes \tau_k\rightarrow \tau_{k \pm 1}$.   Hence $\Gamma^1(\pi_k^{\pm}) =1\otimes \sigma_k^{\pm}$. In particular,  $\Gamma^1(\pi^{+}_k)$ is a surjection. 
Thus we have the following:

\begin{thm}  \label{T:generation}
Let $A_k = \Gamma^1(V_k)$ be the filtration of $\bfA$. 
The action of $\mathfrak p$ on $A_k$ gives a surjection 
\[ 
\mathfrak p \otimes A_k \rightarrow A_{k+1}/A_k.
\] 
In particular, $\bfA=\bfA(G, W_M[s])$ is generated by $\tau_0= (s-2)\otimes W_M[s]$.  
\end{thm}  

\begin{cor} Let $s\geq 2$ be an integer. Then $\bfA(G, W_M[s])$ has a unique irreducible quotient, denoted by $\sigma(G, W_M[s])$. It contains the 
$K$-type $(s-2)\otimes W_M[s]$. 
\end{cor} 

\subsection{Restriction to a subalgebra}  \label{SS:restricting_A} 
Here we recall some results of \cite{thesis} and \cite{restriction}. Assume we have a simple subalgebra $\mathfrak g' \subseteq \mathfrak g$ such that 
\[ 
\langle e, h,f \rangle \subseteq \mathfrak g'.
\] 
In this situation, it is possible to describe the restriction of $\bfA(\mathfrak g, W_M[s])$ to $\mathfrak g'$.  
Using the $h$-grading,  we can define the Heisenberg subalgebra $\mathfrak n'\subseteq \mathfrak g'$, in the same way as for $\mathfrak g$. In particular, 
 $\mathfrak n'=\mathfrak g' \cap \mathfrak n$, $\mathfrak n'= \mathfrak n'_{1}+ \mathfrak n'_{2}$ and $\mathfrak n'_2=\mathbb C\cdot e$. 
We have a filtration of $U(\mathfrak n)$ by $U(\mathfrak n')$-submodules such that 
\[ 
U_1(\mathfrak n) = U(\mathfrak n') \text{ and }   U_k(\mathfrak n) /U_{k-1}(\mathfrak n) \cong U(\mathfrak n') \otimes S^k(\mathfrak n/\mathfrak n'). 
\] 
This filtration gives a $\mathfrak g'$-filtration of the Verma module $V(\mathfrak g, W_M[s])$, which in turn gives a filtration $F_1\subseteq F_2 \subseteq \ldots $  of $\bfA(\mathfrak g, W_M[s])$
such that 
\[ 
F_k/F_{k-1} \cong \bfA(\mathfrak g',  S^k(\mathfrak n/\mathfrak n') \otimes W_M[s+k]). 
\] 
Thus the restriction problem reduces to decomposing $S^k(\mathfrak n/\mathfrak n') \otimes W_M$ into irreducible $M'$-summands. 

We shall almost exclusively use this in the example $G'=\Spin(4,3)$ and $G=\Spin(4,4)$. Then  the embedding $M'\subset M$ is given by 
\[ 
M'\cong \SU(2) \times \Spin(3) \subset \SU(2)\times\SU(2) \times \SU(2)\cong M 
\]  
where $\Spin(3)\cong \SU(2)$ embeds diagonally into last two $\SU(2)$ in $M$.  Moreover,  
\[ 
\mathfrak n'_{1}\cong (1)\otimes (2) \subset (1)\otimes (1) \otimes (1) \cong \mathfrak n_1. 
\] 
Hence 
\[ 
\mathfrak n/\mathfrak n' \cong (1)\otimes (0)  \text{ and } S^k(\mathfrak n/\mathfrak n')\cong (k)\otimes (0) 
\] 
as $M'=\SU(2)\times \Spin(3)$-modules. 

\section{Restriction of quaternionic representations from $\Spin(4,3)$ to $G_2$}  \label{S:B3G2} 

In this section we prove that  the restriction of quaternionic representations $\Spin(4,3)$ to $G_{2,2}$ is finitely generated, for the relevant 
inclusion of $G_{2,2}$ into $\Spin(4,3)$,  where $G_{2,2}$ is the split group of exceptional type $\rG_2$. 

\smallskip 
Recall that the maximal compact subgroup of $\Spin(4,3)$ is 
(a quotient of) $K=\SU_0(2) \times M$ where $M= \Spin(3) \times \SU(2)$.  Let 
$\mathfrak k \oplus \mathfrak p$ be the Cartan decomposition of the complex Lie algebra of $\Spin(4,3)$.  Recall, from Section \ref{S:QG},  that 
$\SU_0(2) \times M$-modules, we have $\mathfrak p = (1) \otimes V_M$, where $V_M=(2)\otimes(1)$.

On the other hand, the maximal compact subgroup of $G_{2,2}$ 
is (a quotient of) 
\[ 
K_2=\SU_s(2) \times \SU_l(2)
\] 
where the two $\SU(2)$ correspond to short and long compact roots, respectively, as the notation indicates. 
Let $\mathfrak k_2\oplus \mathfrak p_2$ be the Cartan decomposition of the complex Lie algebra of $G_{2,2}$. 
Then $\mathfrak p_2= (3)\otimes (1)$ as $\SU_s(2) \times \SU_l(2)$-modules.

We have an embedding  $G_{2,2}\subset \Spin(4,3)$ so that $K_2\subset K$ is given by 
\[ 
\SU_s(2) \times \SU_l(2)  \subseteq \SU_0(2)\times \Spin(3) \times \SU(2)
\] 
where $\SU_s(2)$ embeds diagonally into $\SU_0(2)\times \Spin(3)$, and $\SU_l(2)$ maps to the last factor $\SU(2)$.   
The inclusion $\mathfrak p_2 \subset \mathfrak p$ is given 
\[ 
(3)\otimes (1) \subset (1) \otimes (2)\otimes (1) 
\] 
where $(3)$ embeds as a summand of the first two factors $(1)\otimes (2)= (1)\oplus (3)$, and the map is the identity on the last factor $(1)$.

\begin{thm} \label{T:restriction_g2} 
Assume $s\geq 4$. Let $W_M$ be any finite-dimensional $M$-module.  
Let $\mathfrak g_2=\mathfrak k_2\oplus \mathfrak p_2$ be the complex Lie algebra of $G_{2,2}$. 
Let $U(\mathfrak g_2)$ denote the universal enveloping algebra of $\mathfrak g_2$.  Then 
\[
\bfA(\Spin(4,3), W_M[s])=U(\mathfrak g_2) \cdot ((s-2) \otimes W_M). 
\] 
\end{thm}

\begin{proof} 
We shall prove that the filtration $A_k$ is contained in $U(\mathfrak g_2) \cdot ((s-2) \otimes W_M)$ by induction on $k$. For $k=0$ there is 
nothing to prove. For the induction step, we shall act by $\mathfrak p_2$ on $(s+k-2)\otimes \tau_k$ and prove that we get $A_{k+1}$ modulo $A_k$. 
By Theorem \ref{T:generation} the action of $\mathfrak p$ on $(s+k-2)\otimes \tau_k$ gives a surjection.  
\[ 
\delta  : \mathfrak p\otimes [ (s+k-2) \otimes \tau_k]\rightarrow A_{k+1}/A_k \cong (s+k-1) \otimes \tau_k
\] 
Using the Clebsch--Gordan rule, we can decompose the domain of $\delta$ as a direct sum 
 \[ 
 \mathfrak p\otimes [ (s+k-2) \otimes \tau_k] \cong  (s+k-3) \otimes V_M \otimes \tau_k \oplus  (s+k-1) \otimes V_M \otimes \tau_k. 
 \] 
Clearly the surjection $\delta$ vanishes on the first summand. Hence $\delta= \delta \circ \pi$ where $\pi$ is the projection on the second summand. 
 Thus, to prove the theorem, we need to show that the composition 
\[ 
 \mathfrak p_2\otimes (s+k-2)\otimes \tau_k \rightarrow  \mathfrak p\otimes  (s+k-2) \otimes \tau_k  \rightarrow (s+k-1) \otimes V_M \otimes \tau_k, 
\] 
where the first map is the natural inclusion and the second map the projection $\pi$, is a surjection. 
Note that the factor $\tau_k$ plays no role and we can remove it. Next, 
substitute $ \mathfrak p_2=(3)\otimes (1)$, $ \mathfrak p=(1)\otimes V_M$ and $V_M=(2)\otimes (1)$  where the last factor $(1)$ in both $\mathfrak p_2$ and $V_M$
is a representation of $\SU_l(2)$. We  can remove this factor in all three terms, as well, 
 arriving to a manageable  composition
\[ 
(3)\otimes (s+k-2) \rightarrow (1)\otimes (2) \otimes (s+k-2) \rightarrow (2)\otimes (s+k-1).  
\] 
The surjectivity of the above composition is a subject of the following lemma, which completes the proof of the theorem, and explains the condition $s\geq 4$. 
\end{proof}
\begin{lemma} The composition 
\[ 
f:  (3)\otimes (n) \rightarrow (2)\otimes (1) \otimes (n) \rightarrow (2)\otimes (n+1), 
\] 
obtained by the inclusion  $(3) \subset (2)\otimes (1)$ and the projection $(1)\otimes(n)$ on $(n+1)$, is surjective if $n\geq 2$. 
\end{lemma} 
\begin{proof}  We realize $(n)$ as the space of homogeneous, degree $n$,  polynomials in $x$ and $y$.  Then $(3) \subset (2)\otimes (1)$ is spanned by 
\[ 
y^2\otimes y, \,  2xy \otimes y + y^2 \otimes x, \, 2xy \otimes x + x^2 \otimes y \,  , x^2\otimes x. 
\] 
The projection map $(1)\otimes (n) \rightarrow (n+1)$ is given by the plain multiplication of polynoimials of degree 1 and n.  In order to show that 
$f$ is surjective, we shall argue that in the image of $f$ we get $x^2\otimes q$, $xy\otimes q$ and $y^2\otimes q$ where $q$ 
is any monomial of degree $n+1$. Let $p$ be a monomial of degree $n$. Then 
\[ 
f(x^2 \otimes x \otimes p)=  x^2 \otimes xp \text{ and } f(y^2 \otimes y \otimes p) = y^2\otimes yp. 
\] 
Hence the image of $f$ contains all pure tensors $x^2\otimes q$ and $y^2\otimes q$ for $q\neq y^{n+1}$ and $x^{n+1}$, respectively. 
Let $q=x^ay^b \in (n+1)$ be such that  $a,b\geq 1$. Since $a+b=n+1\geq 3$, either $a$ or $b$ is greater than or equal to $2$.  Assume $a$.  Let 
$p=x^{a-1} y^b$. Then 
\[ 
f((2xy \otimes x + x^2 \otimes y) \otimes p)=  2xy \otimes xp + x^2\otimes yp =  2xy \otimes q + x^2\otimes yp
\] 
is in the image. Since $yp\neq y^{n+1}$, the second summand is in the image of $f$. 
 It follows that $xy\otimes q$ is in the image  of $f$. A similar argument  works if $b\geq 2$. Hence $xy\otimes q$ is in the image of $f$ for any $q$.  Finally, 
 \[ 
 f((2xy \otimes x + x^2 \otimes y) \otimes y^n)=  2xy \otimes xy^n + x^2\otimes y^{n+1}  
 \] 
 hence $x^2\otimes y^{n+1}$ is in the image of $f$. Similarly for $y^2\otimes x^{n+1}$. 
\end{proof}

\section{$D_4$ dual pairs and correspondence in $E_6$}  \label{S:D4} 

In order to determine the correspondence for the dual pairs $K\times B_3$ we shall use another family of dual pairs  $K_1\times D_4$ 
in a see-saw position: 
\begin{equation} \label{eqseesaw_0}
\arraycolsep=1.5pt\def\arraystretch{1.35}
\begin{array}{ccc}
 K & & \Spin(4,4)\\
| & \bigtimes & | \\
 K_1 & & \Spin(4,3).  
\end{array}
\end{equation} 

The family of dual pairs $K_1\times D_4$ was considered by Loke in his thesis \cite{thesis}, and the results were published in \cite{p33}, however 
that paper does not include the dual pair in $E_{6,4}$. The purpose of this section is to fill that gap, and explain the results that we shall need. 

\subsection{The minimal representation}
The maximal compact  subgroup of the quaternionic  adjoint group $\rE_{6,4}$ is $(\SU_0(2) \times \SU(6)/\mu_3)/\mu_2$.
The group of automorphisms of $\rE_{6,4}$ is  $\rE_{6,4} \rtimes \langle \iota \rangle$ where $\iota$ is an outer automorphism of order $2$. 
We pick $\iota$ so that  it commutes with $\SU_0(2)$ and acts as on $\SU(6)$ by the complex conjugation.

The minimal representation $\Vmin$ of $\rE_{6,4}$ is the quaternionic representation $\sigma(G, \bbC[4])$. Its $\SU_0(2)\times \SU(6)$-type decomposition is 
\begin{equation} \label{eqKtypesofVmin}
\Vmin = \bigoplus_{k\geq 0} (k+2) \otimes (k\varpi_3)
\end{equation}
where $\varpi_3=(1,1,1,-1,-1,-1)/2$ is the third fundamental weight for $\SU(6)$.  The minimal representation can be 
extended to $\rE_{6,4} \rtimes \langle \iota \rangle$  in two ways. We fix the one so that $\iota$ acts trivially on the minimal type. 

\subsection{$D_4$ dual pair}  Since $K_1$ is a compact subgroup we will specify the dual pair $K_1 \times D_4$ by 
describing how $K_1$ sits in the maximal compact subgroup of the adjoint 
$E_{6,4}$ which is  isomorphic to $(\SU_0(2) \times \SU(6)/\mu_3)/\mu_2$.  Using this identification, we can identify $K_1$ with the two-dimensional torus 
\[
T = \{ {\mathrm{diag}}(x,x,y,y,z,z) : xyz = 1\} \in \SU(6)/\mu_3.
\]
Let $I$ denote the $2\times 2$ identity matrix.  We highlight three one-parameter subgroups in $T$:  
\[ 
\left[ \begin{array}{ccc} 
x^{2/3} I   &  &  \\
 & x^{-1/3} I &  \\
 & & x^{-1/3} I
\end{array}\right], 
\left[ \begin{array}{ccc} 
y^{-1/3} I   &  &  \\
 & y^{2/3} I &  \\
 & & y^{-1/3} I
\end{array}\right], 
\left[ \begin{array}{ccc} 
z^{-1/3} I   &  &  \\
 & z^{-1/3} I &  \\
 & & z^{2/3} I
\end{array}\right]. 
\] 
The fractional powers make sense since these matrices represent elements of $\SU(6)/\mu_3$. Let $\chi$ be a character of $T$. The restriction of $\chi$ to 
the three one-parameter groups is $x^a$, $y^b$ and $z^c$, respectively, for some integers $a$, $b$ and $c$ such that $a+b+c=0$.  Moreover, 
\[
\chi(\mathrm{diag}(x,x,y,y,z,z))  = x^a y^b z^c.
\] 
Thus characters of $T$ correspond to such triples of integers. Recall that we also have an outer automorphism $\iota$ of $E_{6,4}$ which acts on 
the factor $\SU(6)/\mu_3$ by complex conjugation. Thus, if 
 $\chi$ corresponds to $(a,b,c)$ then $\chi^{\iota}$, the conjugate  of $\chi$ by $\iota$, corresponds to $(-a,-b,-c)$.   
 Hence $\iota$ fixes $\chi$ if and only if $\chi=1$, the trivial character of $T$. 
It is clear now that all irreducible representations of $\tilde T= T\rtimes \langle \iota\rangle$ are 2-dimensional, except the trivial representation, and a non-trivial character $\epsilon$, trivial on $T$.  

The centralizer of $\tilde T$ in $(\SU_0(2) \times \SU(6)/\mu_3)/\mu_2$  is
\[ 
(\SU_0(2) \times \SU(2)\times \SU(2) \times \SU(2))/\mu_2 
\] 
 where the last three $\SU(2)$  embed  in $\SU(6)$ in a block diagonal fashion:
\[ 
(g,h,e) \mapsto 
\left[ \begin{array}{ccc} 
g   &  &  \\
 & h &  \\
 & & e
\end{array}\right]. 
\]
This is the maximal compact subgroup of $D_4\cong \Spin(4,4)$. 

Let $s\geq 2$.  
Recall that  $M\cong \SU(2)\times \SU(2) \times \SU(2)$ for $\Spin(4,4)$ and we have 
 quaternionic representations $\bfA(\Spin(4,4)), (\alpha,\beta,\gamma)[s])$,  with the minimal type $(s-2, \alpha,\beta,\gamma)$.  
 By Theorem \ref{T:generation}, 
this module also has a unique irreducible quotient module denoted by $\sigma(\Spin(4,4)), (\alpha,\beta,\gamma)[s])$, with the same minimal type.  
 
\subsection{Correspondence} 
Suppose restricting $\Vmin$ to the subgroup $\Spin(4,4) \cdot T $ decomposes as
\[
\Vmin = \bigoplus_{a + b + c = 0} \Theta(a,b,c) \otimes \chi(a,b,c).
\]
The next result is taken from \cite[Theorem 5.3.3]{thesis}, except $\Theta^-$ was missed there. A corrected statement is \cite[Theorem 6]{GLPS}.  
\begin{thm} \label{T161} We have 
\begin{enumerate}[(i)]  
\item $\Theta(a,b,c) \cong \Theta(-a,-b,-c)$.

\item Assume  $a \geq 0$, $b \geq 0$, $c < 0$. Then 
\begin{align*}
\Theta(a,b,c) & = \sigma(\Spin(4,4), (b,a,0)[4 + a + b]) \\
\Theta(c,a,b) & = \sigma(\Spin(4,4), (0,b,a)[4 + a + b]) \\
\Theta(b,c,a) & = \sigma(\Spin(4,4), (a,0,b)[4 + a + b]). 
\end{align*}

\item Split  $\Theta(0,0,0)=\Theta^+ \oplus \Theta^-$ by the action of $\iota$ on it. Then 
\[ 
\Theta^+ =  \sigma(\Spin(4,4), (0,0,0)[4]) \text{ and } \Theta^- = \sigma(\Spin(4,4), (0,0,0)[6]).  
\] 
\end{enumerate} 
\end{thm}

\section{The $B_3$ dual pair and correspondence in $E_6$}  \label{S:E6} 

\subsection{The $\Spin(4,3)$ dual pair}  \label{S24}
The group $\rE_{6,4}\rtimes \langle \tau \rangle$ contains a see-saw of dual pairs 

\begin{equation} \label{eqseesaw_1}
\arraycolsep=1.5pt\def\arraystretch{1.35}
\begin{array}{ccc}
 \tilde \rU(2) & & \Spin(4,4)\\
| & \bigtimes & | \\
 \tilde T & & \Spin(4,3).  
\end{array}
\end{equation} 
The embedding $\Spin(4,3) \subset \Spin(4,4)$ is determined by the embedding of 
the corresponding maximal compact subgroups which we now proceed to describe. 
 The maximal compact subgroup of $\Spin(4,3)$, and its embedding into the maximal compact subgroup of $\Spin(4,4)$ is 
\[ 
(\SU_0(2) \times \SU(2) \times \Spin(3))/\mu_2  \subset (\SU_0(2) \times \SU(2)\times \SU(2) \times \SU(2))/\mu_2 
\] 
where the embedding on the first two factors is the identity and $\Spin(3)$ embeds into the last two $\SU(2)$ diagonally.  The centralizer of the maximal compact of 
$\Spin(4,3)$ in $K$ is $\rU(2)$. The subgroup $\SU(2)$  of  $\rU(2)$ embeds as 
\[ 
\left[ \begin{array}{cc} 
  x  &  y \\
  z  & w 
\end{array}\right]
 \mapsto 
\left[ \begin{array}{ccc} 
I  &  &  \\
 & xI  &  yI  \\
 & zI  & wI 
\end{array}\right]. 
\] 
and the center as 
\[ 
\left[ \begin{array}{cc} 
  x  &    \\
    &  x 
\end{array}\right]
 \mapsto 
\left[ \begin{array}{ccc} 
x^{-2/3}I  &  &  \\
 & x^{1/3}I  &    \\
 &  & x^{1/3}I
\end{array}\right]. 
\] 

\begin{prop} \label{P:U(2)_branching} 
The restriction of the irreducible representation of $\rU(2)$ with the highest weight $(m,n)$, $m\geq n$, to $T$ is a direct sum of  $m-n+1$ 
characters corresponding to the triples 
\[ 
(a,b,c)=(-m-n, m,n), (-m-n, m-1, n+1), \ldots , (-m-n, n,m). 
\] 
\end{prop} 
\begin{proof}  The center of $\rU(2)$ acts by the weight $m+n$, so it is clear tha $a=-m-n$.  Hence $b+c=m+n$.   
The $\SU(2)$ weights are $m-n, m-n-2, \ldots , n-m$.  Since 
\[ 
\left[ \begin{array}{ccc} 
x^{-1/3} I   &  &  \\
 & x^{2/3} I &  \\
 & & x^{-1/3} I
\end{array}\right] \cdot 
\left[ \begin{array}{ccc} 
x^{1/3} I   &  &  \\
 & x^{1/3} I &  \\
 & & x^{-2/3} I
\end{array}\right] = 
\left[ \begin{array}{ccc} 
 I   &  &  \\
 & x  I &  \\
 & & x^{-1} I 
 \end{array}\right] 
\] 
it follows that 
\[ 
b-c= m-n, m-n-2, \ldots , n-m. 
\] 
Combining with $m+n=b+c$, we solve for $(b,c)$ as claimed. 
\end{proof}

\subsection{Correspondence} 
Suppose restricting $\Vmin$ to $\Spin(4,3) \cdot \rU(2) $ decomposes as
\[ 
\Vmin  = \bigoplus_{a \geq b} \Theta(a,b) \otimes (a,b) 
\] 
where $(a,b)$ denotes the irreducible representation of $\rU(2) $ with the highest weight $a\geq b$.  
According to \cite[Eqn. (7.1)]{LiDuke}, the nfinitesimal character of $\Theta(a,b)$ is 
\begin{equation} \label{eqinfThetaa2a3}
\lambda(a,b) = \dfrac{1}{2}(a+b+1,a+b-1,a-b+1).
\end{equation}
Here we use the positive root system of $\rB_3$  in $\mathbb R^3$ such that the 
simple roots are $\alpha_1=(1,-1,0)$, $\alpha_2=(0,1,-1)$ and $\alpha_3=(0,0,1)$. The negative highest 
root is $\alpha_0= (-1,-1,0)$. The half sum of positive roots is $\rho=\frac{1}{2}(5,3,1)$. 

\begin{lemma} \label{L321}
The theta lift $\Theta(a,b)$ is a direct sum of finitely many irreducible unitarizable representations.
\end{lemma}

\begin{proof}
We have $\Theta(a,b) \subseteq \Vmin$ which is $\SU_0(2)$-admissible and unitarizable. 
Hence~$\Theta(a,b)$ is an admissible and unitarizable module with infinitesimal character $\lambda(a,b)$. 
This shows that $\Theta(a,b)$ has finite length and unitarizable. 
\end{proof} 

Let $s\geq 2$. The group $\Spin(4,3)$ has quaternionic  representations $\bfA(\Spin(4,3), (m,n)[s])$, where $(m,n)$ is a highest weight for $M=\SU(2)\times \Spin(3)$.  
The two factors of $M$  correspond to the roots $\alpha_1$ and $\alpha_3$, respectively. Specializing (\ref{E:inf_A_general}) to $\Spin(4,3)$, 
 the infinitesimal character of $\bfA(\Spin(4,3), (m,n)[s])$ is 
\begin{equation} \label{inf_A}
\frac{m}{2}\alpha_1 + \frac{n}{2}\alpha_3 + \rho + \frac{s}{2} \alpha_0  =\frac{1}{2}( m-s+5,  -m-s+3, n+1). 
\end{equation}
 We have the following theorem which describes the correspondence between representations of $\tilde \rU(2)$ and $\Spin(4,3)$.  We note that representations of 
 $\rU(2)$ are $\iota$-invariant precisely when the highest weight is $(k,-k)$. Such representations can be extended to $\tilde \rU(2)$ in two ways. We distinguish the two 
 by the action of $\iota$ on the $0$-weight space, and denote them by $(k,-k)^+$ and $(k,-k)^-$.  
\begin{thm} \label{Tmain} Let $(a,b)$ be a highest weight for $\rU(2)$.  
We have: 
\begin{enumerate}[(i)]
\item $\Theta(a,b) \cong \Theta(-b,-a)$.
\item If $a\geq b>0$ then
\[
\Theta(a,b) = \sigma(\Spin(4,3) ,  (0,a-b)[4+a+b]).
\]

\item If $a>  0\geq b$ and $a+b>0$, then 
\[
\Theta(a,b) = \sigma(\Spin(4,3),  (-b, a+b)[4+a]).
\]

\item Let $a\geq 0$. Decompose $\Theta(a,-a)=\Theta(a,-a)^+ \oplus \Theta(a,-a)^-$ by the action of $\iota$. Then 

\begin{align*}
\Theta(a,-a)^+& \subseteq  \sigma(\Spin(4,3), (a,0)[4+a]) \\
\Theta(a,-a)^- & \subseteq  \sigma(\Spin(4,3), (a-1,0)[5+a]).   
\end{align*}
where for $a=0$, the last identity is interpreted as $\Theta(0,0)^-=0$. 

\end{enumerate}
\end{thm}

The proof of the theorem will occupy the rest of this section. The first part is clear (the action of $\tau$).  Furthermore, 
 it is easy to se that $\Theta(a,b)$ is non-zero for any pair $(a,b)$.  However, in the special case 
 $\Theta(a,-a)=\Theta(a,-a)^+ \oplus \Theta(a,-a)^-$ it is considerable trickier to see which of the two summands is non-zero. 
 With some effort it is possible to see that only $\Theta(0,0)^-$ vanishes, and the inclusions in the part (iv) are in fact isomorphisms. 
  Since non-vanishing of theta lifts is not logically 
 necessary for the main goals of the paper, we omit the proof.

\subsection{Proof of (ii)} 
Recall, from Proposition \ref{P:U(2)_branching}, that the restriction of $(a,b)$ 
contains the character corresponding to the triple $(-a-b,a,b)$. 
 By \Cref{T161}(ii)  we have
\[
\Theta(a,b) \subseteq \Theta(-a-b,a,b) = \sigma(\Spin(4,4), (0,b,a)[4 + a + b]).
\]
Recall that $\sigma(\Spin(4,4), (0,b,a)[4 + a + b])$ is a quotient of $\bfA(\Spin(4,4), (0,b,a)[4 + a + b])$. \\
By Section \ref{SS:restricting_A}, this module has an increasing filtration $F_1 \subseteq F_2 \subseteq \ldots$ of Harish-Chandra modules of $\Spin(4,3)$ such that
\[ 
F_{m}/F_{m-1}  =  \sum_{j=0}^b \bfA(\Spin(4,3), (m,a-b + 2j)[4 + a + b + m]).
\]  
We denote the $j$-th summand module on the right hand side of the above equation by~$\bfA(m,j)$. By (\ref{inf_A})
its infinitesimal character is
\begin{equation} \label{eqinfThetaa2a3j}
\lambda(a,b,m,j) = \dfrac{1}{2}(a+b+1+2m,a+b-1,a-b + 1 + 2j).
\end{equation}
Suppose $\Theta(a,b)$ shares an irreducible component with $\bfA(m,j)$.
Then infinitesimal characters $\lambda(a,b)$ in \eqref{eqinfThetaa2a3} and $\lambda(a,b,m,j)$ in \eqref{eqinfThetaa2a3j} are equal.
We compute 
\[
0 = ||\lambda(a,b,m,j)||^2 - ||\lambda(a,b)||^2 = (a+b+m+1)m + \frac{1}{4}(a-b+1+j)j
\]
Since $a \geq  b>0$, $m \geq 0$ and $j \geq 0$, the two terms on the right are non-negative.
The sum of these two terms are zero so each term is zero. We get $m = 0$ and $j = 0$. Hence $\Theta(a,b)$ is a quotient of 
\[
V_1= \bfA(0,0) = \bfA(\Spin(4,3),  (0,a-b) [4 + a + b]).
\]
By \Cref{L321}, $\Theta(a,b)$ is a direct sum of irreducible representations. Hence it has to be isomorphic to the unique irreducible quotient of 
$\bfA(\Spin(4,3),  (0,a-b) [4 + a + b])$. 
This proves~(ii). 

\subsection{Case (iii)} 
Let $c=a+b>0$. Hence, we can write $a=c+k$ and $b=-k$, for some integer $k\geq 0$. 
by Proposition \ref{P:U(2)_branching} the representations of $\rU(2)$ with the highest weight $(c+k,-k)$ are precisely those  that  contain the character of $T$ corresponding to $(-c,c,0)$.  Hence 
\[  
\bigoplus_{k \geq 0} \Theta(c+k,-k) = \Theta(-c,c,0)=  \sigma(\Spin(4,4), (0,0,c)[4 + c]).
\] 
By Section \ref{SS:restricting_A} this module has an increasing filtration $F_1 \subseteq F_2 \subseteq \ldots$ of Harish-Chandra modules of $\Spin(4,3)$ such that
\[ 
F_{k}/F_{k-1} = \bfA(\Spin(4,3), (k,c)[4 + c + k]). 
\] 
Observe that the infinitesimal character of $\Theta(c+k,-k)$ and $\bfA(\Spin(4,3), (k,a)[4 + c + k])$ is  
\[ 
\lambda(c+k, -k)= \dfrac{1}{2}(c+1,c-1,c +2k+1).
\] 
As $k$ varies,  these are different. Hence $\Theta(c+k,-k)$ is a quotient of $\bfA(\Spin(4,3), (k,c)[4 + c + k])$, and we argue as in (ii) to prove that 
$\Theta(c+k,-k)$ is isomorphic to $\sigma(\Spin(4,3), (k,c)[4 + c + k])$. Substitution $c=a+b$ and $k=-b$ gives us 
the claimed $\sigma(\Spin(4,3), (-b,a+b)[4 + a])$.

\subsection{Case (iv)} \label{S34} 
This is similar to the previous case,   using 
\[  
\bigoplus_{a\geq 0} \Theta(a,-a)^+ = \Theta(0,0,0)^+=  \sigma(\Spin(4,4), (0,0,0)[4]).
\]  
and 
\[  
\bigoplus_{a \geq 0} \Theta(a,-a)^- = \Theta(0,0,0)^-=  \sigma(\Spin(4,4), (0,0,0)[6]).
\] 
The  $+$ case  is exactly as  the case (ii) however, for the minus case, we have a filtration with quotients 
\[ 
\bfA(\Spin(4,3), (a-1,0)[5+a])
\] 
where $a\geq 1$ and the infinitesimal character of this is the same as the infinitesimal character of $\Theta(a, -a)$. 
This completes the proof of \Cref{Tmain}.

\section{$B_3$ dual pair and correspondence in $E_7$}  \label{S:E7} 

\subsection{The minimal representation}
The minimal representation $\Vmin$ of  the adjoint $\rE_{7,4}$ is the quaternionic representation $\sigma(\rE_{7,4},\bbC[6])$.
This representation has $\SU_0(2) \times \Spin(12)$-types 
\[
\Vmin = \bigoplus_{k=0}^\infty (k+4) \otimes (k \omega_6)
\]
where $\omega_6 = \frac{1}{2}(1,1,1,1,1,1)$ is the highest weight of a half-spin representation on $\Spin(12)$.

\subsection{Local theta lifts} 
The group $\rE_{7,4}$ contains a see-saw of dual pairs 

\begin{equation} \label{eqseesaw}
\arraycolsep=1.5pt\def\arraystretch{1.35}
\begin{array}{ccc}
\rSp(2)\times  \rSp(1) & & \Spin(4,4)\\
| & \bigtimes & | \\
 \rSp(1)^3 & & \Spin(4,3).  
\end{array}
\end{equation}
We shall not need to know how these dual pairs sit in $\rE_{7,4}$, so we omit that part. The vertical inclusions above are obvious. 
The highest weight of an irreducible representation of $\rSp(2)$ is $(a,b)$ where $a$ and $b$ are integers such that 
$a\geq b\geq 0$ in the usual $\rC_2$ root system realization.  
The highest weight (and the corresponding representation of $\rSp(2)\times\rSp(1)$) will be denoted by $(a,b;c)$, 
Suppose restricting $\Vmin$ (the minimal representation of quaternionic $E_7$) 
 to $\Spin(4,3) \times [\rSp(2) \times \rSp(1)] $ decomposes as
\[ 
\Vmin  = \bigoplus_{(a,b;c)} \Theta((a,b; c))  \otimes (a,b; c)
\] 
According to \cite[Eqn. (3.8)]{LiDuke} $\Theta((a,b;c))$ has infinitesimal character
\begin{equation} \label{eqinfThetc2c1}
\lambda(a,b; c) = \dfrac{1}{2}(a+b+3,a-b+1,c+1).
\end{equation}
(We caution  the reader that Li's result is in terms of the $\rB_2$ root system, as $\rSp(2)\cong \Spin(5)$.)

\begin{thm} With $M=\SU(2)\times \Spin(3)$,  
\begin{itemize} 
\item If $c\leq a-b$ then 
\[ 
 \Theta((a,b; c))\cong \sigma(\Spin(4,3), (b,c)[6+a]),  
\] 

\item If $a-b \leq c\leq a+b$ then 
\[ 
 \Theta((a,b; c))\cong \sigma (\Spin(4,3), \left(\frac{a+b-c}{2}, a-b \right) \left[6+\frac{a+b+c}{2} \right] ),  
\] 
\item If $a+b <c$ then $ \Theta((a,b; c))=0$.  

\end{itemize} 
\end{thm} 
\begin{proof} 
Recall we have another see-saw 
\begin{equation} \label{eqseesaw1}
\arraycolsep=1.5pt\def\arraystretch{1.35}
\begin{array}{ccc}
\rSp(3) & & \Spin(4,3)\\
| & \bigtimes & | \\
\rSp(2)\times\rSp(1) & & \rG_{2,2}
\end{array}
\end{equation}
where $\rG_{2,2}$ denotes the split Lie group of type $\rG_2$. The $\rSp(3)$-types that appear in the minimal
representation have the highest weight $(\alpha,\beta,\gamma)$ such that $\alpha=\beta+\gamma$ \cite{HPS}.  
Now, using the Wallach--Yacobi branching, Proposition \ref{P:C_n branching} , one can check that the branching to $\rSp(2)\times\rSp(1)$ only gives $(a,b;c)$ such that $c\leq a+b$. 
This gives a verification of the third bullet, and non-vanishing in other cases. To give the upper bound of the theta lift, we 
use the see-saw dual pair $\rSp(1)^3 \times \Spin(4,4)$ and the lift in this case obtained by Loke, to figure out~ $\Theta((a,b;c))$. 
To that end we need the following, an easy consequence of Proposition~\ref{P:C_n branching}.  

\begin{lemma} Fix $d \geq 0$. Representations of $\rSp(2)$ that, when restricted to $\rSp(1)\times\rSp(1)$, contain 
$(d,0)$ (or equivalently $(0, d)$) are the representations with the highest weight $(d+k,k)$ where $k=0,1, \ldots$. 
Moreover, the multiplicity of $(0,d)$ is always 1.
\end{lemma} 

Hence, by a see-saw argument, 
\[  
\bigoplus_{k \geq 0} \Theta(d+k,k;c) = \Theta(0,d,c)
\] 
(or $=\Theta(d,0,c)$).  Here $\Theta(0,d,c)$ is the lift to $\Spin(4,4)$ of the representation of $\rSp(1)^3$ of the highest weight 
$(0,d,c)$.  

Assume that $d\geq c$. Then, using \cite[Thm. 12.4.1]{p33},
\[ 
\Theta(d,0,c)=\sigma(\Spin(4,4), (0,c,0)[6+d]), 
\] 
and 
\[ 
\Theta(0,d,c)=\sigma(\Spin(4,4), (c,0,0)[6+d]).
\] 
Since the restriction to $\Spin(4,3)$ is the same, 
this is  saying that $\Spin(3)$ (of $B_3$) sits in the first two $\SU(2)$ in $M$ of $\Spin(4,4)$. By 
Section \ref{SS:restricting_A}
$\bfA(\Spin(4,4), (0,c,0)[6+d])$ (and also $\bfA(\Spin(4,4), (0,c,0)[6+d])$ have a $\Spin(4,3)$-flitration $F_k$ such that 
\[ 
F_k/F_{k-1}\cong \bfA(\Spin(4,3), (k,c)[6+d+k])
\] 
where $(k,c)$ is the highest weight for $M=\SU(2) \times \Spin(3)$.  
By (\ref{inf_A}) the infinitesimal character of $\bfA(\Spin(4,3), (k,c)[6+d+k])$ is 
\[ 
 \dfrac{1}{2}(d+2k+3,d+1,c+1) 
\] 
and this is precisely the infinitesimal character $\lambda(d+k,k; c)$ of $\Theta(d+k,k;c)$ by  (\ref{eqinfThetc2c1}). 
As these characters are all different a $k$ varies, it follows that 
\[ 
\Theta(d+k,k;c)\cong \sigma(\Spin(4,3), (k,c)[6+d+k]), 
\] 
if $d\geq c$.  Now assume $c\geq d$. Then 
\[ 
\Theta(d,0,c)=\sigma(\Spin(4,4), (0,d,0)[6+c])
\] 
and 
\[ 
\Theta(0,d,c)=\sigma(\Spin(4,4), (d,0,0)[6+c]).
\] 
Now $\bfA(\Spin(4,4), (0,d,0)[6+c])$ (and also $\bfA(\Spin(4,4), (0,d,0)[6+c])$ have a $\Spin(4,3)$-flitration $F_k$ such that 
\[ 
F_k/F_{k-1}\cong \bfA(\Spin(4,3), (k,d)[6+c+k]). 
\] 
 The infinitesimal character of this subquotient is (we are  using the Weyl group action, as needed, to pick a nicer representative for comparison)
\begin{equation}\label{eq_one} 
 \dfrac{1}{2}(c+2k+3,d+1,c+1). 
 \end{equation} 
 On the other hand, the infinitesimal character of $\Theta(d+l,l; c)$ is 
 \begin{equation} \label{eq_two}
  \lambda(d+l,l; c) = \dfrac{1}{2}(d+2l+3,d+1,c+1). 
 \end{equation} 
 The infinitesimal character in (\ref{eq_one}) is equal to the one in (\ref{eq_two}) precisely if $c+2k=d+2l$ hence $k=l+(d-c)/2$. 
  This proves the second bullet, and provides an independent verification of the third, since $k$ is non-negative, forcing $l$ to be at least 
  $(c-d)/2$.  
\end{proof}

\section{ $B_3$ dual pair and correspondence in $E_8$}  \label{S:E8} 

The treatment in this case is somewhat different, since there is no correspondence of infinitesimal characters. 
Instead, we shall identify theta lifts by computing their minimal $K$-types.  

\subsection{The minimal representation}
The minimal representation $\Vmin$ of $\rE_{8,4}$ is the quaternionic representation $\sigma(\rE_{8,4},\bbC[10])$.
This representation has $ \SU_0(2) \times \rE_7$-types 
\[
\Vmin = \oplus_{k=0}^\infty (k+8) \otimes (k \varpi_7)
\]
where $\omega_7$ is the highest weight of the irreducible miniscule 56 dimensional representation of~$\rE_7$. 

\subsection{Local theta lifts} \label{S32}
The group $\rE_{8,4}$ contains a see-saw of dual pairs 

\begin{equation} \label{eqseesaw2}
\arraycolsep=1.5pt\def\arraystretch{1.35}
\begin{array}{ccc}
\Spin_9 & & \Spin(4,4)\\
| & \bigtimes & | \\
\Spin_8 & & \Spin(4,3).  
\end{array}
\end{equation}

Let $\lambda$ be a highest weight for $\Spin_8$. Let $(\lambda)$ denote the corresponding irreducible representation of $\Spin_8$. 
Restricting $\Vmin$ to the dual pair $\Spin_8 \times \Spin(4,4)$ gives
\[
\Vmin = \bigoplus_{\lambda} (\lambda) \otimes \Theta(\lambda)
\]
where the sum is over all highest weights $\lambda$ for $\Spin_8$.  The factor
$\Theta(\lambda)$ is a quaternionic representation of $\Spin(4,4)$.  This lift is computed in \cite[Theorem 1.4.1]{p33}.
We extract the relevant information from the theorem that we need later. 

\begin{thm} \label{TSpin8Spin44}
Let $\lambda = (a,b,c,d)$ be a highest weight of $\Spin_8$. Then 
\[
\Theta(\lambda) = (b-c+1) \cdot \bfA(\Spin(4,4), (a-b,c+d,c-d)[10+a+b]).
\]
Here $\bfA(\Spin(4,4), (a-b,c+d,c-d)[10+a+b])$ is the quaternionic discrete series representation of $\Spin(4,4)$ 
with the same infinitesimal character $\lambda + \rho(\rD_4)$ as $(\lambda)$.
\end{thm}

\subsection{}
Restricting $\Vmin$ to the dual pair $\Spin(9) \times \Spin(4,3)$ gives
\[
\Vmin = \bigoplus_{w} (w) \otimes \Theta(w)
\]
where the sum is over all highest weights $w$ for $\Spin(9)$. Here $(w)$ denotes the corresponding irreducible representation of $\Spin(9)$. The factor
$\Theta(w)$ is a an admissible quaternionic representation of $\Spin(4,3)$.

Fix $b\geq d\geq 0$ two half-integers, congruent modulo $1$.  Let $\lambda=(b,b,d,d)$, considered a highest weight for $\Spin(8)$. By the Gelfand--Zetlin rule, 
irreducible representations of $\Spin(9)$ containing $(\lambda)$ are $(w)$ where 
$w=(a,b,c,d)$.  (Observe that there are $b-d+1$ choices for $c$.) 
Thus, by the see-saw, the restriction of $\Theta(\lambda)$ to $\Spin(4,3)$ is isomorphic to a sum 
\begin{equation} \label{eq_seesaw_8} 
\Theta(\lambda) = \bigoplus_{w} \Theta(w)
\end{equation} 
 over all $\Spin(9)$-highest weights $w=(a,b,c,d)$ with $b$ and $d$ fixed. Now, by  \Cref{TSpin8Spin44}, we have 
\[ 
\Theta(\lambda)= (b-d+1)  \bfA(\Spin(4,4), (0,2d,0)[10+2b]). 
\] 
Since this module is unitarizable, and the restriction to $\Spin(4,3)$ is admissible, 
the filtration in Section \ref{SS:restricting_A} becomes a direct sum. 
\begin{align} \label{eq_filtration_8} 
\bfA(\Spin(4,4), (0,2d,0)[10+2b])  = &\bigoplus_{n=0}^\infty \bfA(\Spin(4,3), (n) \otimes (2d)  [10+ 2b + n])\\ 
 =  &\bigoplus_{a\geq b}^\infty   \bfA(\Spin(4,3), (a-b) \otimes (2d) [10+a+b]. 
\end{align}
Observe that the $\Spin(4,3)$-modules appearing in the above sum are automatically irreducible, since they are unitarizable and
have unique irreducible quotients.  
Combining the equations (\ref{eq_seesaw_8}) and (\ref{eq_filtration_8}), for any $b\geq d$, we have 
\begin{equation} \label{eq_combined_8} 
 \bigoplus_{w} \Theta(w)= (b-d+1) \bigoplus_{a\geq b}^\infty   \bfA(\Spin(4,3), (a-b) \otimes (2d) [10+a+b], 
\end{equation} 
where the sum, on the left hand side, is taken over all highest weights $(a,b,c,d)$ with $b$ and $d$ fixed.  

\begin{thm} \label{TthetaSpin9w}
Let $w = (a,b,c,d)$ be a highest weight for $\Spin(9)$. Then
\[
\Theta(w) = \bfA(\Spin(4,3), (a-b, 2d)[10 + a+b]).
\]
This is an irreducible (unitarizable) quaternionic discrete series representation with infinitesimal character
\[
\left(a  +  \frac{7}{2}, b + \frac{5}{2}, d+ \frac{1}{2} \right).
\]

\end{thm}

\begin{proof} In view of the equation (\ref{eq_combined_8}) it suffices to prove the inclusion 
\begin{equation}\label{E:inclusion} 
\Theta(w) \supseteq \bfA(\Spin(4,3), (a-b, 2d)[10 + a+b]).
\end{equation} 
We shall do that by computing the minimal type of $\Theta(w)$. We are looking for a minimal $\SU_0(2) \times \SU(2)\times \Spin(3)$-type. 
Recall that the maximal compact of $\rE_{4,4}$ is $\SU_0(2) \times \rE_7$ and we have further inclusions 
\begin{equation}  \label{E:inclusions}
\rE_7 \supset \SU(2) \times \Spin(12) \supset \SU(2) \times \Spin(3) \times \Spin(9). 
\end{equation} 
We need a branching rule (see \cite[Lemma 6.1.1]{p33}).

\begin{lemma} \label{LrestE7D6} Consider $\Spin(12)$-highest weights of the form 
$\lambda(x,y,z) =(x,y,z,z,z,z)$. 
We have
\[
{\mathrm{Res}}^{\rE_7}_{\SU_0(2) \times \Spin(12)} (k \omega_7) 
= \bigoplus_{x+y=k} (x-y) \otimes (\lambda(x,y,z)). \ \ \qed
\]
\end{lemma}

\begin{lemma}
Let $w = (a,b,c,d)$ be a highest weight for $\Spin(9)$, the group sitting in $\rE_7$ as in (\ref{E:inclusions}).  
Let $k$ be the minimal integer such that $(k \varpi_7) \supseteq (w)$.
Then $k = a+b$.  In that case, 
\[ 
\Hom_{\Spin(9)}( (w), ((a+b) \varpi_7) \cong (a-b)\otimes (2d) 
\] 
as $\SU(2) \times \Spin(3)$-module, where $(2d)$ is the irreducible representation of $\Spin(3)$ with the highest weight $2d$.  Here we are using the convention that 
the weights of $\Spin(3)$ are integral (as opposed to half integral).  
\end{lemma}

\begin{proof}
Suppose the $\Spin(12)$-module $(\lambda(x,y,z))$ contains $(w)$.
By the Gelfand--Zetlin rule \cite{GT}, this happens if and only if
\[
x \geq a \geq z, \, y \geq b \geq z, \, z \geq c \geq z \geq d \geq 0.
\]
The minimal value of $k=x+y$ satisfying the above inequalities is $k=a+b$ for $x=a$ and $y=b$. Moreover, since $z$ must be $c$, we see that $(w)$ is contained in the summand 
\[
(a-b) \otimes (\lambda(a,b,c))
\]
of $((a+b)\varpi_7)$, from \Cref{LrestE7D6}.

It remains to compute the $\Spin(3)$-module 
\[ 
\Hom_{\Spin(9)}( (w), (\lambda(a,b,c))). 
\] 
 To that end, we shall use a see-saw argument and the restriction formula to $\Spin(10)\times \Spin(2)$. 
  Note that $\Spin(2)$ is the maximal torus in $\Spin(3)$.  
By \cite{GT}, the $\Spin(10)$-modules contained in $(\lambda(a,b,c))$ and containing $(w)$ have highest weights 
$(a,b,c,c,e)$ where $|e|\leq d$.  On each of these, $\Spin(2)$ acts by the weight $e$ by Proposition \ref{P:D_n branching}. (The convention there is that 
the weights for $\Spin(2)$ are half-integers.) 
Thus the above is the $(2d+1)$-dimensional irreducible representation of $\Spin(3)$, as claimed. 
\end{proof}

Now we can show the inclusion (\ref{E:inclusion}). Using the factorization $K=\SU_0(2)\times M$, the smallest $\SU_0(2)$-type of $\Theta((a,b,c,d))$ is $(a+b+8)$ and on 
$M=\SU(2)\times \Spin(3)$ we have $(a-b)\otimes (2d)$, the minimal type of $\bfA(\Spin(4,3), (a-b, 2d)[10 + a+b])$. Hence $\Theta(w)$ contains~$\bfA(\Spin(4,3), (a-b, 2d)[10 + a+b])$.  
\end{proof}

\section{Dual pair in $F_4$}  \label{S:F4}

\subsection{The minimal representation}
The minimal representation $\Vmin$ of $\rF_{4,4}$ is the quaternionic representation $\sigma(\rF_{4,4},\bbC[3])$, of the nonlinear 2-fold cover. 
This representation has $K$-types ($K = \SU_0(2) \times \rSp(3)$)
\[
\Vmin = \oplus_{k=0}^\infty (k+1) \otimes (k \varpi_3)
\]
where $\omega_3=(1,1,1)$ is the third fundamental weight of~$\rSp(3)$. This is a representation of a non-linear 2-fold cover $\tilde{\rF}_{4,4}$ of $\rF_{4,4}$.  

\subsection{Local theta lifts} 
The (linear) group $\rF_{4,4}$ contains a see-saw of dual pairs 

\begin{equation} \label{eqseesaw_2}
\arraycolsep=1.5pt\def\arraystretch{1.35}
\begin{array}{ccc}
\mathrm{O}(2) & & \Spin(4,4)\\
| & \bigtimes & | \\
K_4 & & \Spin(4,3).  
\end{array}
\end{equation}
where $K_4$ is the  center of $\Spin(4,4)$. It is a Klein 4-group. In $\tilde{\rF}_{4,4}$ the right hand side of the see-saw pair is replaced by two-fold covers (non-linear groups) 
while the left hand side splits. We describe the splitting. 
The embedding of the maximal compact of (the cover of) $\Spin(4,4)$ into the maximal compact of $\tilde F_{4,4}$  is given by 
\[ 
\SU_0(2) \times \SU(2) \times \SU(2) \times \SU(2) \cong   \SU_0(2) \times \rSp(1) \times \rSp(1) \times \rSp(1)  \subset \SU_0(2) \times \rSp(3).
\] 
We give the splitting of $K_4$ in the product of the three $\rSp(1)$, 
as triples of  $\pm1\in \rSp(1)$ whose product is $1$.
The group $S_3$ of outer automorphisms of $\Spin(4,4)$ permutes the three $\rSp(1)$.  
A choice of $S_2 \subset S_3$ picks $\Spin(3,4) \subset \Spin(4,4)$. We fix the choice so that $S_2$ fixes the first $\rSp(1) \subset \rSp(3)$.   
 Hence the factor $\Spin(3)$ of the maximal compact subgroup of $\Spin(4,3)$ embeds diagonally in the last two $\rSp(1) \times \rSp(1)  \subset \rSp(2)$. 
 The centralizer of $\Spin_3\cong \rSp(1)$ in $\rSp(2)$ is $\mathrm{O}(2)$, a member of a Howe dual pair, however, this group clearly does not contain 
 the Klein 4 group. We need to twist the embedding of $g\in \mathrm{O}(2)$ into $\rSp(2)$ by the sign $\det(g)$ in the first $\rSp(1)$, to get the 
  embedding of $\mathrm{O}(2)$ that contains the Klein 4 group. 

Our goal is to compute the correspondence for the dual pair 
$\mathrm{O}(2) \times \Spin(4,3)$. The strategy here is similar to other cases: we shall use the above see-saw.
We have the following, Theorem~4.6.1 \cite{p33}  or Theorem 4.3.2 in \cite{thesis}:

\begin{thm} 
The restriction of $\Vmin$ to $\Spin(4,4)$ is a sum of four  irreducible representations: $\sigma(\Spin(4,4), (0,0,0)[3])$, fixed by $S_3$, and 
\[ 
 \sigma(\Spin(4,4), (1,0,0)[4]) \oplus  \sigma(\Spin(4,4), (0,1,0)[4]) \oplus \sigma(\Spin(4,4),  (0,0,1)[4]). 
\] 
permuted transitively by $S_3$.  
\end{thm}

Recall that the characters of $\mathrm{SO}(2)$ are parameterized by integers, and for every natural number $n$  there is is unique 2-dimensional 
representation $\tau(n)$ of $\mathrm O(2)$ whose restriction to $\mathrm{SO}(2)$ is $(-n)\oplus (n)$.
With respect to $\Spin(4,3) \times \mathrm{SO}(2)$ we can decompose 
\[ 
\Vmin=\bigoplus_{n\in \mathbb Z} \Theta(n) \otimes (n) 
\] 
where clearly $\Theta(n)\cong\Theta(-n)$, the lift of the two-dimensional representation $\tau(n)$ of $\mathrm O(2)$, if $n\neq 0$, and 
$\Theta(0)=\Theta(0)^+ \oplus \Theta(0)^-$, the sum of the lifts of the trivial and non-trivial characters of $\mathrm O(2)$. 

\begin{prop}  The infinitesimal character of $\Theta(n)$ is $\frac{1}{2}(n,2,1)$.    
\end{prop} 
\begin{proof}  The centralizer of  $(1,-1,-1)\in  K_3$ in $\rF_{4,4}$ is $\Spin(4,5)$. This is an easy check, for example the centralizer of $(1,-1,-1)$ in 
$\SU(2) \times \rSp(3)$ 
\[ 
\SU(2) \times \rSp(1)\times \rSp(2) \cong \Spin(4) \times \Spin(5),
\]  
the maximal compact in $\Spin(4,5)$. The restriction of $\Vmin$ to $\Spin(4,5)$ (its 2-fold cover to be precise) decomposes as a sum 
\[ 
\Vmin\cong \Vmin^+ \oplus \Vmin^-
\] 
where $(1,-1,-1)$ act on $\Vmin^{\pm}$ by $\pm$.  
By Theorem 4.3.1 in \cite{thesis}  $\Vmin^+$ and $\Vmin^-$ are irreducible. In fact, they are 
small representations studied in \cite{LS1} and \cite{LS2}.  

Also, observe that  $(1,-1,-1)\in \mathrm{O}(2)$ (in fact the central element). Thus 
\[ 
\Spin(4,3) \times \mathrm{O}(2) \subseteq \Spin(4,5), 
\] 
and we can decompose $\Vmin^+$ and $ \Vmin^-$ under the action of $\Spin(4,3) \times \mathrm{SO}(2)$. 
Clearly $\Vmin^+$ picks up $\Theta(n)$ for $n$ even and $\Vmin^-$ picks up $\Theta(n)$ for $n$ odd. The matching of infinitesimal characters for 
these correspondences is given by Theorem 1.2  in \cite{LS1}, that is, the infinitesimal character of $\Theta(n)$ is $(n,2,1)/2$.  
\end{proof}

\begin{thm} Let $k\geq 0$. With the identification $M\cong \SU(2) \times \Spin(3)$, 
\begin{itemize} 
\item $\Theta(2k) = \sigma(\Spin(4,3), (k,0)[k+3])$, where $k\neq 0$.  
\item  $\Theta(2k+1) = \sigma(\Spin(4,3), (k,1)[k+4])$. 
\item $\Theta(0)^+ = \sigma(\Spin(4,3), (0,0)[3])$ and $\Theta(0)^- \cong  \sigma(\Spin(4,3), (0,0)[5])$
\end{itemize} 
\end{thm}  
\begin{proof}  Using the $\Spin(4,4)$ dual pair, each character of the Klein group gives a see-saw identity. Two of them are 
\[  
\bigoplus_{k\geq 0} \Theta(2k+1) \cong \sigma(\Spin(4,4), (0)\otimes (1,0)[4]) \cong  \sigma(\Spin(4,4), (0)\otimes (0,1)[4]). 
\] 
Now each $\bfA(\Spin(4,4), (0)\otimes (1,0)[4])$ and $\bfA(\Spin(4,4), (0)\otimes (0,1)[4])$ has a $\Spin(4,3)$-filtration (see Section  \ref{SS:restricting_A}) 
whose subquotients are 
\[ 
F_k/F_{k-1} \cong \bfA(\Spin(4,3), (k)\otimes (1)[k+4]. 
\] 
The infinitesimal character of this representation is (\ref{inf_A})
\[ 
\frac{1}{2}(1, -2k -1 , 2)\sim \frac{1}{2}(2k+1 , 2,1)
\] 
and this is exactly the infinitesimal character of $\Theta(2k+1)$.  Thus $\Theta(2k+1)$ must be a unitarizable quotient of $\bfA(\Spin(4,3), (k)\otimes (1)[k+4]$.  
Hence the second bullet.  The first bullet is proved in the same way, using the see-saw identity 
\[ 
\sigma(\Spin(4,4), (0)\otimes (0,0)[3])\cong \Theta(0)^+ \oplus_{k>1} \Theta(2k). 
\] 
We leave the details to the reader. It remains to determine $\Theta(0)^-$.  We use the remaining see-saw identity, 
\[ 
\sigma(\Spin(4,4), (1)\otimes (0,0)[4])\cong \Theta(0)^- \oplus_{k>1} \Theta(2k). 
\] 
Since we already know $\Theta(2k)$ for $k>0$, it remains to isolate $\Theta(0)^-$ using the infinitesimal character.  
By Section  \ref{SS:restricting_A}, $\bfA(\Spin(4,4), (1) \otimes (0,0)[4])$ has a $\Spin(4,3)$-filtration whose subqutients are 
\begin{equation} \label{eq_two_terms} 
F_k/F_{k-1}=\bfA(\Spin(4,3), (k+1) \otimes (0)[4 +k]) \oplus  \bfA(\Spin(4,3), (k-1) \otimes (0)[4 +k]). 
\end{equation} 
These the summands two have infinitesimal characters 
\[ 
\frac{1}{2}(2, -2k -2, 1) \text{ and } \frac{1}{2}(0, -2k, 1) 
\] 
respectively. Only  $\bfA(\Spin(4,3), (0) \otimes (0)[5])$  (the second summand in (\ref{eq_two_terms}) for $k=1$) has the infinitesimal character of 
$\Theta(0)$, that is,  $\frac{1}{2}(0, 2, 1)$. 
\end{proof}

\section{Representations of  $\PU(2,1)$ and $G_2$}  \label{S:unitary}
In this section, we list the representations of $G_2$ and $\PU(2,1)$ whose lifts we would like to compute. Our main goal is to compute $\Theta(\pi')$ when $\pi'$ is a unitarizable representation of $G_2$ with integral infinitesimal character. It is known that such representations are realized as $A_{\mathfrak q}(\lambda)$ modules, so our first step is to write them down. It will be crucial to know their $K$-types.
\subsection{$K$-types for $G_2$}
\label{MKTG2}
We will refer to the following picture:
\begin{center}
	\scalebox{0.9}{
  \begin{tikzpicture}[
    -{Straight Barb[bend,
       width=\the\dimexpr10\pgflinewidth\relax,
       length=\the\dimexpr12\pgflinewidth\relax]},
  ]
    \foreach \i in {0, 1, ..., 5} {
      \draw[thick] (0, 0) -- (\i*60:2);
      \draw[thick] (0, 0) -- (30 + \i*60:3.5);
    }
	\foreach \x in {0,1,2,3,4,5,6,7,8,9,10}
	\foreach \y in {0,1,2,3,4,5}
	\filldraw(\x,1.75*\y) circle (1pt);
    \filldraw[red](2,0) circle (2pt);
    \filldraw[red](0,3.5) circle (2pt);
    \node[above] at (0,3.5) {$(2,-1,-1)$};
    \node[below right] at (2, 0) {$(0,1,-1)$};
	 \filldraw[blue](3,-1.75) circle (2pt);
	\filldraw[blue](3,1.75) circle (2pt);
   	\filldraw[blue](1,1.75) circle (2pt);
   	\filldraw[blue](-1,1.75) circle (2pt);
  \end{tikzpicture}
}
  \end{center}
Aside from the usual set of coordinates $(a,b,c)$ with $a+b+c=0$, we will be using the coordinates $(x,y)=(b-c,a)$ corresponding to the above picture. (Thus $(a,b,c)=(y,\frac{x-y}{2},-\frac{x+y}{2})$.)

We choose a set of positive roots; these are marked with (red and blue) circles. The compact roots are $(0,1,-1)$ and $(2,-1,-1)$ (red). The maximal compact group $K$ is a product of 
two $\SU(2)$ corresponding to the short and long compact roots. In particular
\[
\rho_c = (1,0,-1).
\]
Our goal is to describe the $K$-types of $A_{\mathfrak{q}}(\lambda)$'s. We use Vogan--Zuckerman \cite[Theorem 5.3]{VoganZuckerman} to determine the cone of $K$-types for a given pair $(\lambda,\mathfrak{q})$. We start with $\lambda = (a,b,c)$ in the positive Weyl chamber, and consider its Weyl orbit. 
\subsubsection{Regular $\lambda$}
\label{MKTG2_reg}
 For regular $\lambda$, there is only one choice of $\theta$-stable parabolic, and we get three different representations, corresponding to the Weyl group action on $\lambda$. Take $(a,b,c)$ such that $a>b>0$.
 
 \begin{figure}[h]
 	\begin{center}
 		\scalebox{0.5}{
 			\begin{tikzpicture}[
 				-{Straight Barb[bend,
 					width=\the\dimexpr10\pgflinewidth\relax,
 					length=\the\dimexpr12\pgflinewidth\relax]},
 				]
 				\foreach \i in {0, 1, ..., 5} {
 					\draw[thick] (0, 0) -- (\i*60:2);
 					\draw[thick] (0, 0) -- (30 + \i*60:3.5);
 				}
 				\draw[thick, red] (10,4*1.75) -- (9,5*1.75);
 				\draw[thick, red] (10,4*1.75) -- (13,3*1.75);
 				\draw[thick, green] (1,7*1.75) -- (-2,8*1.75);
 				\draw[thick, green] (1,7*1.75) -- (4,8*1.75);
 				\draw[thick, blue] (13,1*1.75) -- (14,2*1.75);
 				\draw[thick, blue] (13,1*1.75) -- (14,0);
 				\foreach \x in {0,1,2,3,4,5,6,7,8,9,10,11,12,13,14}
 				\foreach \y in {0,1,2,3,4,5,6,7,8}
 				\filldraw(\x,1.75*\y) circle (1pt);
 				\filldraw[green](1,3*1.75) circle (2pt);
 				\filldraw[red](4,2*1.75) circle (2pt);
 				\filldraw[blue](5,1.75) circle (2pt);
 				
 				\node[below left] at (1, 3.5*3.5) {\Large{B}};
 				\node[below left] at (10, 2*3.5) {\Large{A}};
 				\node[left] at (13, 0.5*3.5) {\Large{C}};	 
 		\end{tikzpicture}}
 	\end{center}
 	\caption{The cones of $K$-types for  regular $\lambda$. Pictured: $(a,b,c)=(2,1,-3)$ and $A_\mathfrak{q}(\lambda)$ modules A (red), B (green), and C (blue).}
 	\label{fig1}
 \end{figure}

\noindent \textbf{Rep A} $\lambda = (a,b,c)$. Here we have $\rho = (2,1,-3)$, $\rho(\mathfrak{u}\cap\mathfrak{p}) = (1,1,-2)$. The infinitesimal character is 
\[
\lambda + \rho = (a+2,b+1,c-3),
\] 
and the minimal type is 
\[
\mu = \lambda + 2\rho(\mathfrak{u}\cap\mathfrak{p}) = (a+2,b+2,c-4).
\]
In $(x,y)$-coordinates, this is $(b-c+6,a+2)$.

\noindent \textbf{Rep B} $\lambda = (-c,-b,-a)$. Here we have
$\rho =(3,-1,-2)$, $\rho(\mathfrak{u}\cap\mathfrak{p})=(2,-1,-1)$. The infinitesimal character is
\[
\lambda + \rho = (3-c,-b-1,-a-2)
\]
and the minimal type is
\[
\mu = \lambda + 2\rho(\mathfrak{u}\cap\mathfrak{p}) = (-c+4,-b-2,-a-2).
\]
In $(x,y)$-coordinates, this is $(a-b,-c+4)$.

\noindent \textbf{Rep C} $\lambda=(b,a,c)$. Here $\rho =(1,2,-3)$, $\rho(\mathfrak{u}\cap\mathfrak{p})=(0,2,-2)$, so the infinitesimal character is  
\[
\lambda + \rho = (b+1,a+2,c-3)
\]
and the minimal type is 
\[
\mu =  \lambda + 2\rho(\mathfrak{u}\cap\mathfrak{p}) =(b,a+4,c-4),
\]
i.e. $(a-c+8,b)$ in $(x,y)$-coordinates.

 \subsubsection{The wall $a=b$}
 \label{MKTG2_wall1}
 We consider the case $(a,b,c)=(a,a,-2a)$. We have multiple possibilities for $\frakq$, and we get four different representations.
 
 \begin{figure}[h]
 	\begin{center}
 		\scalebox{0.5}{
 			\begin{tikzpicture}[
 				-{Straight Barb[bend,
 					width=\the\dimexpr10\pgflinewidth\relax,
 					length=\the\dimexpr12\pgflinewidth\relax]},
 				]
 				\foreach \i in {0, 1, ..., 5} {
 					\draw[thick] (0, 0) -- (\i*60:2);
 					\draw[thick] (0, 0) -- (30 + \i*60:3.5);
 				}
 				\draw[thick, red] (9,3*1.75) -- (8,4*1.75);
 				\draw[thick, red] (9,3*1.75) -- (12,2*1.75);
 				\draw[thick, violet] (10,2*1.75) -- (11,3*1.75);
 				\draw[thick, violet] (10,2*1.75) -- (13,1*1.75);
 				\draw[thick, blue] (11,1*1.75) -- (12,2*1.75);
 				\draw[thick, blue] (11,1*1.75) -- (12,0);
 				\draw[thick, green] (0,6*1.75) -- (0,7*1.75);
 				\draw[thick, green] (0,6*1.75) -- (3,7*1.75);
 				\foreach \x in {0,1,2,3,4,5,6,7,8,9,10,11,12,13}
 				\foreach \y in {0,1,2,3,4,5,6,7}
 				\filldraw(\x,1.75*\y) circle (1pt);
 				
 				\node[below] at (0,6*1.75) {\Large{B}};
 				\node[below left] at (9,3*1.75) {\Large{A}};
 				\node[left] at (10,2*1.75) {\Large{AC}};	 
 				\node[left] at (11,1*1.75) {\Large{C}};	 
 		\end{tikzpicture}}
 	\end{center}
 	\caption{The cones of $K$-types for cases Ia and IIa on the wall $a=b$: the $A_\mathfrak{q}(\lambda)$ modules A (red), AC (purple), C (blue) with $\lambda = (1,1,-2)$, and B (green) with $\lambda =(2,-1,-1)$}
 	\label{fig2} 
 \end{figure}

\noindent \textbf{Case Ia:} $a = b>0$, i.e. $\lambda = (a,a,-2a)$. Here we have three possible $\mathfrak{q}$'s; accordingly, there are three possible representations. The infinitesimal character is
\[
(a+2,a+1,-2a-3)
\] 
and the minimal types (corresponding to different choices of $\mathfrak{q}$) are given by $ \lambda + 2\rho(\mathfrak{u}\cap\mathfrak{p})$:

\noindent \textbf{Rep A\phantom{C}} $\quad \mu = (a+2,b+2,c-4) =(a+2,a+2,-2a-4)$\\
\textbf{Rep AC}  $\quad \mu = (a+1, b+3, c-4) = (a+1,a+3,-2a-4)$\\
\textbf{Rep C\phantom{A}} $\quad \mu = (a, b+4, c-4) = (a,a+4,-2a-4)$.

\noindent In $(x,y)$-coordinates, this is
\[
(3a+6,a+2),\quad (3a+7,a+1), \quad (3a+8,a).
\]
   
\noindent \textbf{Case IIa:} $\lambda = (2a,-a,-a)$. We get one new representation; the infinitesimal character is 
\[
(2a+3,-a-1,-a-2)
\] 
and the minimal type

\noindent \textbf{Rep B}$\quad \mu = \lambda + 2\rho(\mathfrak{u}\cap\mathfrak{p}) = (2a+4,-a-2,-a-2)$.

\noindent In $(x,y)$-coordinates, this is $(0,2a+4)$.

 \subsubsection{The wall $b=0$}
 \label{MKTG2_wall2} We consider the case $(a,b,c)=(a,0,-a)$. Again, we get four representations. 
 
 \begin{figure}[h]
 	\begin{center}
 		\scalebox{0.55}{
 			\begin{tikzpicture}[
 				-{Straight Barb[bend,
 					width=\the\dimexpr10\pgflinewidth\relax,
 					length=\the\dimexpr12\pgflinewidth\relax]},
 				]
 				\foreach \i in {0, 1, ..., 5} {
 					\draw[thick] (0, 0) -- (\i*60:2);
 					\draw[thick] (0, 0) -- (30 + \i*60:3.5);
 				}
 				\draw[thick, green] (1,5*1.75) -- (-2,6*1.75);
 				\draw[thick, green] (1,5*1.75) -- (4,6*1.75);
 				\draw[thick, orange] (4,4*1.75) -- (3,5*1.75);
 				\draw[thick, orange] (4,4*1.75) -- (7,5*1.75);
 				\draw[thick, red] (7,3*1.75) -- (6,4*1.75);
 				\draw[thick, red] (7,3*1.75) -- (10,2*1.75);
 				\draw[thick, blue] (10,0) -- (11,0);
 				\draw[thick, blue] (10,0) -- (11,1.75);
 				\foreach \x in {0,1,2,3,4,5,6,7,8,9,10,11}
 				\foreach \y in {0,1,2,3,4,5,6}
 				\filldraw(\x,1.75*\y) circle (1pt);
 				
 				\node[below] at (1,5*1.75) {\Large{B}};
 				\node[below] at (4,4*1.75) {\Large{AB}};
 				\node[left] at (7,3*1.75) {\Large{A}};	 
 				\node[left] at (10,0) {\Large{C}};	 
 		\end{tikzpicture}}
 	\end{center}
 	\caption{The cones of $K$-types for cases Ib and IIb on the wall $b=0$. Here we see $A_\mathfrak{q}(\lambda)$ modules A (red), AB (orange), B (green) with $\lambda = (1,0,-1)$, and C (blue) with $\lambda =(0,1,-1)$.}
 	\label{fig3} 
 \end{figure}

\noindent \textbf{Case Ib:} $\lambda = (a,0,-a)$. We have three representations corresponding to different choices of $\mathfrak{q}$. The infinitesimal character is
\[
(a+2,1,-a-3)
\] 
and the minimal types are

\noindent
\textbf{Rep A\phantom{B}} $\quad \mu = (a+2,2,-a-4)$\\
	\textbf{Rep AB}  $\quad \mu = (a+3,0,-a-3)$\\
	\textbf{Rep B\phantom{A}} $\quad \mu =(a+4,-2,-a-2)$.

\noindent In $(x,y)$-coordinates, these are
\[
(a+6,a+2)\quad (a+3,a+3), \quad (a,a+4).
\]

\noindent \textbf{Case IIb:} $\lambda = (0,a,-a)$. As before, we get only one other representation. The infinitesimal character is
\[
(1,a+2,-a-3)
\] 
and the minimal type is

\noindent \textbf{Rep C} $\quad \mu=(0, a+4, -a-4)$,

\noindent that is, $(2a+8,0)$ in $(x,y)$-coordinates.

\subsection{$K$-types for $\PU(2,1)$}
\label{MKTU2}
We now do the same for $\PU(2,1)$.
\begin{center}
  \begin{tikzpicture}[
  ]
    \foreach \i in {0, 1, ..., 5} {
      \draw[thick,->] (0, 0) -- (\i*60:2);
    }
    \draw[dashed] (3,1.75) to (-3,-1.75);
    \draw[dashed] (6,-3.5) to (-6,3.5);
    \draw[dashed] (0,-3.5) to (0,3.5);
	\foreach \x in {-5,-4,-3,-2,-1,0,1,2,3,4,5}
	\filldraw(\x,5.25) circle (1pt);
	\foreach \x in {-5,-4,-3,-2,-1,0,1,2,3,4,5}
	\filldraw(\x,3.5) circle (1pt);
	\foreach \x in {-3,-2,-1,0,1,2,3,4,5}
	\filldraw(\x,1.75) circle (1pt);
	\foreach \x in {0,1,2,3,4,5}
	\filldraw(\x,0) circle (1pt);
	\foreach \x in {3,4,5}
	\filldraw(\x,-1.75) circle (1pt);
    \filldraw[red](1,1.75) circle (2pt);
    \node[above] at (1,1.75) {$(1,0,-1)$};
        \filldraw[blue](2,0) circle (2pt);
       \filldraw[blue](-1, 1.75) circle (2pt);
    \node[below right] at (2, 0) {$(0,1,-1)$};
        \node[above left] at (-1, 1.75) {$(1,-1,0)$};
  \end{tikzpicture}
  \end{center}
Again, we have the usual set of coordinates $(a,b,c)$ with $a+b+c=0$, as well as the coordinates used in the picture: $(x,y)=(b-c,a)$. (Thus $(a,b,c)=(y,\frac{x-y}{2},-\frac{x+y}{2})$.) When talking about $K$-types, $(a,b,c)$ corresponds to $(a,c)$ in the standard notation for $\rU(2)$ highest weights. We want to describe the $K$-types of $A_{\mathfrak{q}}(\lambda)$'s for any given $\lambda$ in the positive $K$-chamber. Since we have at most three additional representations with the same infinitesimal character, we include them in the picture.

Take $(a,b,c)$ such that $a>b>c$.

\begin{figure}[h]
	\begin{center}
		\scalebox{0.6}{
			\begin{tikzpicture}
				\foreach \i in {0, 1, ..., 5} {
					\draw[thick,->] (0, 0) -- (\i*60:2);
				}
				\draw[dashed] (3,1.75) to (-3,-1.75);
				\draw[dashed] (6,-3.5) to (-6,3.5);
				\draw[dashed] (0,-3.5) to (0,3.5);
				\foreach \x in {-6,-5,-4,-3,-2,-1,0,1,2,3}
				\filldraw(\x,5.25) circle (1pt);
				\foreach \x in {-5,-4,-3,-2,-1,0,1,2,3,4}
				\filldraw(\x,3.5) circle (1pt);
				\foreach \x in {-3,-2,-1,0,1,2,3,4,5}
				\filldraw(\x,1.75) circle (1pt);
				\foreach \x in {0,1,2,3,4,5,6}
				\filldraw(\x,0) circle (1pt);
				\foreach \x in {3,4,5, 6,7}
				\filldraw(\x,-1.75) circle (1pt);
				\filldraw[black](1,1.75) circle (1pt);
				\filldraw[black](2,0) circle (1pt);
				\filldraw[black](-1, 1.75) circle (1pt);
				\draw[thick, green] (-4+0.6, 3.5-0.35) to (-5-0.2,3.5-0.35);
				\draw[thick,green, ->] (-5-0.2,3.5-0.35) -- (-6.2 - 0.3, 5.25+0.175);
				\draw[thick,green, ->] (-4+0.6, 3.5-0.35) -- (-5+0.6 - 0.3,5.25 + 0.175);
				\draw[thick,orange,->] (-2-0.2, 3.5-0.35) -- (-3-0.2-0.3,5.25+0.175);
				\draw[thick,orange] (-2-0.2, 3.5-0.35) to (0.6,3.5-0.35);
				\draw[thick, orange, ->] (0.6,3.5-0.35) -- (-0.7, 5.25+0.175);
				\draw[thick,red,->] (2-0.2, 3.5-0.35) -- (1-0.2-0.3,5.25+0.175);
				\draw[thick,red,->] (2-0.2, 3.5-0.35) -- (4+0.2,3.5-0.35);
				\draw[thick,violet,->] (3-0.6, 1.75+0.35) -- (5,1.75+0.35);
				\draw[thick,violet,->] (3.8, -0.35) -- (6.2,-0.35);
				\draw[thick, violet] (3-0.6, 1.75+0.35) to (3.8, -0.35);
				
				\draw[thick,blue,->] (4.4, -1.75+0.35) -- (6.8,-1.75+0.35);
				\draw[thick,blue] (4.4, -1.75+0.35) to (4.8,-1.75-0.35);
				\draw[thick, blue,->] (4.8, -1.75-0.35) -- (7.2,-1.75-0.35);
				\draw[thick] (-1-0.6, 1.75+0.35) to (3.8-4, -0.35);
				\draw[thick] (-1-0.6, 1.75+0.35) to (1.2, 1.75+0.35);
				\draw[thick] (1.2, 1.75+0.35) to (2.6,-0.35);
				\draw[thick] (-0.2, -0.35) to (2.6,-0.35);
				
				\filldraw(0,3.5) circle (3pt);
				\node[above left] at (0,3.5) {\Large{AB}};
				\filldraw(-4,3.5) circle (3pt);
				\node[above left] at (-4,3.5) {\Large B};
				\filldraw(2,3.5) circle (3pt);
				\node[above right] at (2,3.5) {\Large A};				
				\filldraw(3,1.75) circle (3pt);				
				\node[below right] at (3,1.75) {\Large AC};					
				\filldraw(5,-1.75) circle (3pt);
				\node[right] at (5,-1.75) {\phantom{I}\Large C};					
				\filldraw(1,1.75) circle (3pt);
				\node[left] at (1,1.75) {\Large F\phantom{I}};					
				
			\end{tikzpicture}
		}
	\end{center}
	\caption{The $K$-types for $(a,b,c)=(1,0,-1)$. The one thickened dot in each representation is the type $\tau$ that we lift in Section \ref{S:correspondence}.}
	\label{fig4}
\end{figure}

 \noindent The representations labeled A,B,C in Figure \ref{fig4} are $A_{\mathfrak q}(\lambda)$ modules; we list their minimal types (depicted as thickened dots):

\noindent \textbf{Rep A} $\lambda=(a,b,c)$. Here we have $\rho = (1,0,-1)$ and $2\rho(\mathfrak{u}\cap\mathfrak{p})=(1,0,-1)$. The infinitesimal character is 
\[
\lambda + \rho = (a+1,b,c-1),
\] 
and the minimal type is 
\[
\mu = \lambda + 2\rho(\mathfrak{u}\cap\mathfrak{p}) = (a+1,b,c-1).
\]
In $(x,y)$-coordinates, this is $(b-c+1,a+1)$. In standard $\rU(2)$ coordinates, this is $(a+1,c-1)$.

\noindent \textbf{Rep B} $\lambda = (a,c,b)$. Here we have $\rho = (1,-1,0)$ and $2\rho(\mathfrak{u}\cap\mathfrak{p})=(1,-2,1)$. The infinitesimal character is 
\[
\lambda + \rho = (a+1,c-1,b),
\] 
and the minimal type is 
\[
\mu = \lambda + 2\rho(\mathfrak{u}\cap\mathfrak{p}) = (a+1,c-2,b+1).
\]
In $(x,y)$-coordinates, this is $(c-b-3,a+1)$. In standard $\rU(2)$ coordinates, this is $(a+1, b+1)$.

\noindent \textbf{Rep C} $\lambda=(b,a,c)$. Here we have $\rho = (0,1,-1)$ and $2\rho(\mathfrak{u}\cap\mathfrak{p})=(-1,2,-1)$. The infinitesimal character is 
\[
\lambda + \rho = (b,a+1,c-1),
\] 
and the minimal type is 
\[
\mu = \lambda + 2\rho(\mathfrak{u}\cap\mathfrak{p}) = (b-1,a+2,c-1).
\]
In $(x,y)$-coordinates, this is $(a-c+3,b-1)$. In standard $\rU(2)$ coordinates, this is $(b-1,c-1)$.

\noindent It is now easy to list the types for representations AB, AC, and F

\noindent \textbf{Rep AB} Depending on parity, the minimal type is either the midpoint between the minimal types of A and B, or the point next to it (as in Figure \ref{fig4}). In $(x,y)$ coordinates, this is
\[
\begin{cases}
	(-1, a+1) \quad &\text{if }a \text{ is even};\\
	(0, a+1) \quad &\text{if }a \text{ is odd}.
\end{cases}
\]
In standard $\rU(2)$ coordinates, this is
\[
\begin{cases}
	(a+1, -\frac a2) \quad &\text{if }a \text{ is even};\\
	(a+1, -\frac{a+1}{2}) \quad &\text{if }a \text{ is odd}.
\end{cases}
\]
\noindent \textbf{Rep AC} Similar to AB, the minimal type in this case is (almost) the midpoint between minimal types of A and C, depending on parity:
\[
\begin{cases}
	(2- \frac{3c}{2}, -\frac{c}{2}) \quad &\text{if }c \text{ is even};\\
	(\frac{3-3c}{2}, \frac{1-c}{2}) \quad &\text{if }c \text{ is odd}.
\end{cases}
\]
In standard $\rU(2)$ coordinates, this is
\[
\begin{cases}
	(-\frac{c}{2}, c-1) \quad &\text{if }c \text{ is even};\\
	(\frac{1-c}{2}, c-1) \quad &\text{if }c \text{ is odd}.
\end{cases}
\]
\noindent \textbf{Rep F} The finite-dimensional representation always contains the type obtained from the minimal type of Rep A by adding $(-1,-1)$ (in $(x,y)$ coordinates). Thus we get $(b-c,a)$ in $(x,y)$ coordinates, which is $(a, c)$ in standard $\rU(2)$ coordinates.

\noindent These computations still hold when $\lambda$ is on the wall. Just like in \S \ref{MKTG2}, when $\lambda$ is on the wall we get one additional $A_\mathfrak{q}(\lambda)$: AC when $a=b$, and AB when $b=c$.

\section{The correspondence}
\label{S:correspondence}
In this section we use the results obtained above to deduce fine-grained information about the correspondence between $\PU(2,1)$ and $G_2$.
Our main goal is to compute $\Theta(\pi')$ for all representations $\pi'$ of $G_2$ listed in \S\ref{MKTG2}. Given a representation $\pi'$ of $G_2$, our strategy is to work out  $\dim\Hom_{K'}(F_{\tau}, \pi')$ for suitably chosen types $\tau$.  

Recall, from Section  \ref{S:D4}, that the minimal representation is a module for the disconnected group $\rE_{6,4} \rtimes \langle \iota \rangle$. 
The centralizer of $G_2$ in $\rE_{6,4} \rtimes \langle \iota \rangle$
 is $\PU(2,1) \rtimes \langle \iota \rangle$ where $\iota$ acts by complex conjugation on $\PU(2,1)$.  
Thus $\Theta(\pi')$ is in fact a representation of the disconnected group. Its infinitesimal character is determined by the infinitesimal character of $\pi'$. More precisely, 
we can define the center of the enveloping algebra of $\PU(2,1) \rtimes \langle \iota \rangle$ to be the $\iota$-invariant elements of the usual center of the enveloping algebra of 
$\PU(2,1)$. It is isomorphic to the center of the enveloping algebra of $G_2$.  Assume that the infinitesimal character of $\pi'$ is $(a+2,b+1,c-3)$. 
Since the correspondence of infinitesimal characters is the identity with our choice of coordinates (see Remark \ref{R:tripoints}), the infinitesimal character of $\Theta(\pi')$ is $(a+2,b+1,c-3)$. 
The center of the enveloping algebra of $\PU(2,1)$ is slightly larger and we can decompose 
\[ 
\Theta(\pi') = \Theta(\pi')^+ \oplus \Theta(\pi')^- 
\] 
so that the center of the enveloping algebra of $\PU(2,1)$  acts by infinitesimal characters $(a+2,b+1,c-3)$ and $(-c-3,-b-1,-a-2)$, respectively, on the two summands. 
 Each of the two summands is a $\PU(2,1)$-module, and $\iota$ permutes the two. Clearly, it suffices to determine $\Theta(\pi)^+$.

 Let $\lambda = (a,b,c)$. The infinitesimal character of $\pi'=A_\mathfrak{q}(\lambda)$ on the $G_2$ side is $(a+2,b+1,c-3)$. Thus, constituents of $\Theta(\pi')^+$  are 
 representations of $\PU(2,1)$ with the infinitesimal character $(a+2,b+1,c-3)$.   (Subtracting $\rho$, this means we have $\lambda = (a+1,b+1,c-2)$ for the $A_\mathfrak{q}(\lambda)$ modules of $\PU(2,1)$ 
 possibly appearing here.)
For each of the six representations with this infinitesimal character, we take the type $\tau$ listed in \S\ref{MKTU2}, compute $F_\tau$ using Theorem \ref{Tmain}, and finally, compute the restriction of $F_\tau$ to $K_2$. This will produce a segment of $K_2$ types parallel to the $x$-axis in Figure \ref{fig1}. We list these types in Table \ref{table1}:

\renewcommand{\arraystretch}{1.5}
\begin{table}[h]
\begin{center}
		\caption{The map $\tau \mapsto F_\tau$ for types $\tau$ listed in \S\ref{MKTU2}}
				\label{table1}
	\begin{tabular}{ |c|c|c|c| } 
		\hline 
		Rep & $\tau$ & $F_\tau$ & Restriction of $F_\tau$ to $K_2$\\
		\hline
		A & $(a+2,c-3)$ & $(5-c, a+2, b+1)$ & $(4+a,a+2) \dots (6+b-c,a+2)$\\
		B & $(a+2,b+2)$ & $(6+a+b, 0 , a-b)$ & $(6+2b,0) \dots (6+2a,0)$\\ 
		C & $(b,c-3)$ & $(5-c,b,a+3)$ & $(2+b,b)\dots(8+a-c,b)$\\ 
		\multirow{2}{1.5em}{AB} & $(a+2, -\frac{a+1}{2})$ & $(4+a, \frac{a+1}{2}, \frac{a+3}{2})$ & $(\frac{a+5}{2},\frac{a+3}{2})\dots(\frac{3a+11}{2},\frac{a+3}{2})$\\ 
		 & $(a+2, -\frac{a+2}{2})$ & $(4+a, \frac{a+2}{2}, \frac{a+2}{2})$ & $(\frac{a+6}{2}, \frac{a+2}{2}) \dots (\frac{3a+10}{2},\frac{a+2}{2})$ \\ 
		\multirow{2}{1.5em}{AC} & $(\frac{2-c}{2}, c-3)$ & $(5-c,\frac{2-c}{2}, \frac{4-c}{2} )$ & $(\frac{6-c}{2},\frac{2-c}{2}) \dots (\frac{14-3c}{2},\frac{2-c}{2})$\\ 
		 & $(\frac{3-c}{2}, c-3)$ & $(5-c,\frac{3-c}{2}, \frac{3-c}{2} )$ & $(\frac{7-c}{2},\frac{3-c}{2})\dots(\frac{13-3c}{2},\frac{3-c}{2})$ \\ 
		F & $(a+1,c-2)$ & $(4-c, a+1, b+1)$ & $(3+a, a+1)\dots (5+b-c,a+1)$\\ 		
		\hline
	\end{tabular}
\end{center}
\end{table}

\noindent We are now ready to compute the lifts:

\begin{thm}\label{G2toPU3regular} For $X \in \{A,B,C,AB,AC\}$, let $\pi_X'$ denote the $A_\mathfrak{q}(\lambda)$ module of $G_2$ from \S\ref{MKTG2_reg}; similarly, let $\pi_X$ denote the corresponding representations of $\PU(2,1)$ from \S\ref{MKTU2}. Then 
	\begin{itemize}
	\item For $X \in \{A,C,AC\}$, we have $\Theta(\pi'_X)\cong \pi_X \oplus\pi_X^{\iota}$, if non-zero.
	\item $\Theta(\pi'_B)=0$ and $\Theta(\pi'_{AB})=0$.
	\end{itemize}
\end{thm}

In the first bullet, since $\pi_X$ is not isomorphic to its complex conjugate $\pi_X^{\iota}$, we see that  $\Theta(\pi'_X)$ is an irreducible $\PU(2,1)\rtimes\langle \iota\rangle$-module, if not zero. 
We summarize the correspondence in Table \ref{table3} below:
\begin{table}[h]
	\caption{The lifts. Row AB happens when $b=0$, row AC when $a=b$.}
	\label{table3}
	\begin{center}
		\begin{tabular}{ |c|c|c|c| } 
			\hline 
			Rep $\pi'$ of $G_2$ & Minimal type of $\pi_X'$ & $\Theta(\pi_X')^+$, if non-zero & Minimal type of $\Theta(\pi_X')^+$\\
			\hline
			$\pi_A'$ & $(6+b-c,a+2)$ & $\pi_A$ & $(a+2,c-3)$\\
			$\pi_B'$ & $(a-b,4-c)$ & $-$ & $-$\\			
			$\pi_C'$ & $(8+a-c,b)$ & $\pi_C$ & $(b,c-3)$\\			
			$\pi_{AB}'$ & $(a+3, a+3)$ & $-$ & $-$\\ 
			$\pi_{AC}'$ & $(3a+7, a+1)$ & $\pi_{AC}$ & $(a+1, -a-3)$\\ 
			\hline
		\end{tabular}
	\end{center}
\end{table}
\begin{proof}[Proof of Theorem \ref{G2toPU3regular}]
	The proof is a simple application of Theorem \ref{T:introduction}:
		\[
		\dim \Hom_K(\Theta(\pi'), \tau) \leq \dim \Hom_{K'} (F_{\tau}, \pi'). 
		\]
	We begin by lifting $\pi'_A$. From $\S \ref{MKTG2_reg}$ we know its minimal type,
	\[
	(6+b-c,a+2).
	\]
	We know that constituents of $\Theta(\pi'_A)^+$ are representations from \S \ref{MKTU2}, listed in Table \ref{table1}. Taking $\tau=\tau_A$, we observe that the highest type contained in $F_\tau$ is precisely the minimal type of $\pi_A$. Thus
	\[
	\dim \Hom_K(\Theta(\pi'_A), \tau_A) \leq \dim \Hom_{K'} (F_{\tau_A}, \pi'_A) = 1,
	\]
	which shows $\dim \Hom_K(\Theta(\pi'_A), \pi_A) \leq  1$.
	
	\noindent To eliminate other representations, we simply show that $\dim \Hom_{K'} (F_{\tau}, \pi'_A)=0$ for other $\tau$'s in the table above. To that end, we recall the cone of $K$-types for $\pi_A'$ (see Figure \ref{fig4}). It is determined by lines
	\[
 	k\dots\ x - (6+b-c) + 3(y-(a+2)) = 0 \quad \text{and} \quad l\dots\ x - (6+b-c) + y-(a+2) = 0.
	\]
	We now observe:
	\begin{itemize}
	\item $\tau = \tau_B$, the maximal type in $F_\tau$ is below the line $k$. Indeed,
	\[
	6+2a - (6+b-c) - (a+2) = a-b+c -2  = -2b-2 < 0.
	\]
	\item $\tau = \tau_C$, the maximal type in $F_\tau$ is below the line $l$:
	\[
	8+a-c - (6+b-c) +3(b- (a+2)) = 2+a-b +3b-3a-6  = -2a+2b-4 < 0.
	\]
	\item $\tau = \tau_{AB}$, the maximal type in $F_\tau$ is below the line $l$:
	\[
	\frac{3a+10}{2} - (6+b-c) +3(\frac{a+2}{2}- (a+2)) = \frac{3a}{2} - b+ c -1  - \frac{3a}{2} -3 = -b+c-4 < 0.
	\]
	(the computation for $a$ odd is similar).
	\item $\tau = \tau_{AC}$, the maximal type in $F_\tau$ is below the line $l$:
	\[
	\frac{14-3c}{2} - (6+b-c) +3(\frac{2-c}{2}- (a+2)) = 1 - \frac{c}{2} - b  - \frac{3c}{2} -3a -3 = -a+b-2 < 0.
	\]
	(the computation for $c$ odd is similar).
	\item $\tau = \tau_{F}$, the maximal type in $F_\tau$ is below the line $k$:
	\[
	5+b-c - (6+b-c) a+1 - (a+2) = -2 < 0.
	\]
	\end{itemize}
	To summarize, each of the remaining candidate representations contains a type $\tau$ for which $\dim \Hom_{K'} (F_{\tau}, \pi'_A)=0$. This proves the claim for $\pi_A'$. The lifts of the remaining representations are entirely analogous, so we leave them as an exercise for the reader, and refer to Figure \ref{fig7} for a visual explanation of the proof.
\end{proof}

\begin{figure}[h]
	\includegraphics[scale=0.1]{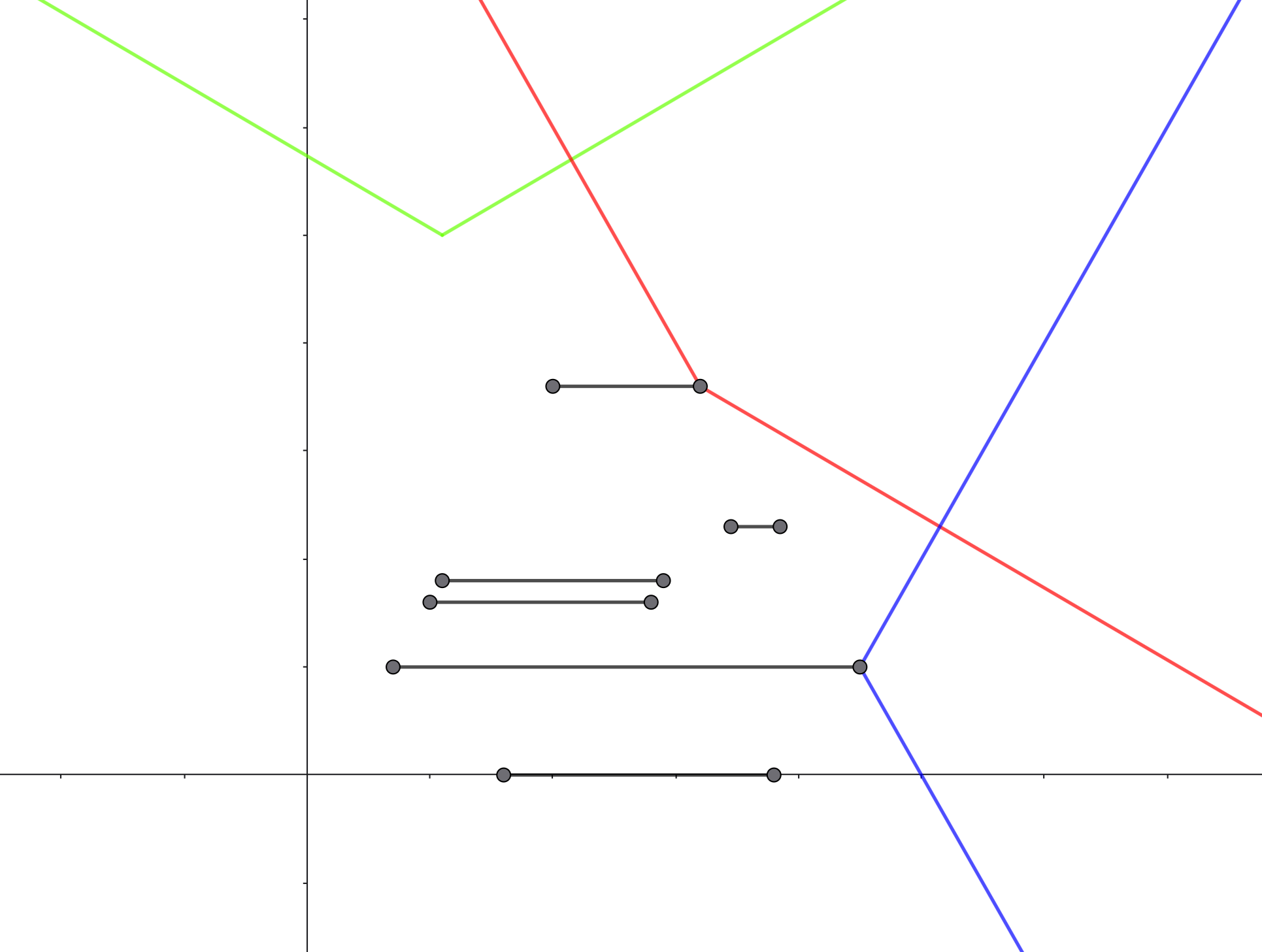}
	\caption{A 'visual proof' of Theorem \ref{G2toPU3regular}: Most $F_\tau$'s miss the cones entirely, whereas $F_{\tau_A}$ (resp.\ $F_{\tau_C}$) intersects the cone of $\pi_A'$ (resp.\ $\pi_C'$) precisely in the minimal type.}
	\label{fig7}
\end{figure}

\begin{rmk} \label{R:tripoints} 
The $G_2$ representation B (which does not appear in this correspondence) appears as a theta lift from the compact group $\PU(3)$, as shown by \cite[Theorem 5.2]{HPS}. 
Indeed, the representation B with minimal type  $(a-b,2-c)$ is the theta lift of the finite-dimensional representation of $\PU(3)$ with the highest weight $(a,b,c)$, 
and its complex conjugate, the finite-dimensional representation of $\PU(3)$ with the highest weight $(-c,-b,-a)$. 
 This fact alone fixes the correspondence of infinitesimal characters between $G_2$ and $\PU(3) \rtimes \langle \iota \rangle$: 
 it is the identity in our coordinates, as claimed. The same is true for the $\PU(2,1)$ representations of type B, they do not show up in this correspondence, 
 but come from the compact $G_2$.  
\end{rmk}

The same mapping $\tau \mapsto F_\tau$ of minimal types that we used in Theorem \ref{G2toPU3regular} can be used to establish the correspondence of $A_\mathfrak{q}(\lambda)$ modules in the other direction. Indeed, let $\pi$ be an $A_\mathfrak{q}(\lambda)$ module on the $\PU(2,1)$ side. If we assume that the lift is unitarizable, then it has to be one of the $A_\mathfrak{q}(\lambda)$ modules on the $G_2$ side. As we saw earlier, the restriction of $F_\tau$ to the maximal compact group of $G_2$ intersects at most one of the cones (see Figure \ref{fig7}), which determines the lift. Rather than stating exhaustive results, we state a sample result in the interesting case when $\lambda$ is on the wall $a=b$:

\begin{thm} \label{T:main_wall} Let $\lambda = (a,b,c)=(a,a,-2a)$ with $a>0$. 
For $X \in \{A,B,C,AC\}$, let $\pi_X$ denote the $A_\mathfrak{q}(\lambda)$ module of $\PU(2,1)$ from \S\ref{MKTU2}. Similarly, let $\pi_X'$ denote the corresponding representations of $G_2$ from \S\ref{MKTG2_wall1} with $\lambda = (a-1, a-1,2-2a)$. Let $\theta(\pi)$ denote the maximal semi-simple and unitarizable quotient of $\Theta(\pi)$. 
	\begin{itemize}  
	\item For $X \in \{A,C,AC\}$, we have $\theta(\pi_X) \cong \pi_X'$.
	\item $\theta(\pi_B)=0$. 
	\end{itemize} 
	\end{thm}

\section{Appendix - some branching rules} \label{S:branching} 

In this section we gather some useful branching rules for classical groups, the known branching from $\rSp(n)$ to $\rSp(n-1) \times \rSp(1)$ and a 
remarkably similar branching from $\Spin(n+1)$ to $\Spin(n-1) \times \Spin(2)$ which appeared in \cite{thesis}. We also include a result on the branching 
from $F_4$ to $B_4$ that we have derived as a consequence of results in this paper. 
Remarkably, this branching is still open, that is, there is no formula that does not involve an alternating sum of large numbers with huge cancelations. 
In this section $\Lambda(G)$ denotes the set of highest weights of finite dimensional representations  of the group $G$.

\subsection{Branching rule from $\rSp(n)$ to $\rSp(n-1) \times \rSp(1)$} 

This formulation is due to Wallach--Yacobi \cite{WY}.  Let $\lambda \in \Lambda(\rSp(n))$. Then 
$\lambda=(x_1, \ldots , x_n)$ where $x_1, \ldots , x_n$ is a descending sequence of non-negative integers.  
 Let $(\lambda)$ denote the finite dimensional representation of  $\rSp(n)$ with that highest weight. 
Let $\mu=(y_1, \ldots , y_{n-1}) \in \Lambda(\rSp(n-1))$. Then $(\mu)$ appears in the restriction of $(\lambda)$ if and only if $\mu$ 2-step interlaces $\lambda$, 
that is, 
\[ 
x_1 \geq y_1 \geq x_3, \, x_2 \geq y_2 \geq x_4, \ldots, x_{n-1} \geq y_{n-1} \geq 0. 
\] 
Let 
\begin{equation} \label{E:ordering} 
z_1 \geq z_2 \geq \ldots \geq z_{2n-1}
\end{equation} 
be the ordering of $x_i$ and $y_j$, that is, $z_1=x_1$, $z_2=y_1$ or $x_2$, depending which one is greater etc.  

\begin{prop} \label{P:C_n branching} 
Let $\lambda \in \Lambda(\rSp(n))$ and $\mu \in \Lambda(\rSp(n-1))$.  Assume that $\mu$ 2-step interlaces $\lambda$. Then, 
as $\rSp(1)$-modules, 
\[ 
\Hom_{\rSp(n-1)} ( (\mu), (\lambda)) \cong (z_1-z_2) \otimes (z_3-z_4) \otimes \cdots \otimes (z_{2n-1})  
\] 
where $z_i$ are as in (\ref{E:ordering}).  
\end{prop} 

\subsection{Branching rule from $\Spin(2n+1)$ to $\Spin(2n-1) \times \Spin(2)$} 
Here we have branching rule which is remarkably similar to the one for $\rSp(n)$ groups.  Let $\lambda \in \Lambda(\Spin(2n+1))$. Then 
  $\lambda=(x_1, \ldots , x_n)$ where $x_1, \ldots , x_n$ is a descending sequence of non-negative half 
 integers, however, $x_i \equiv x_j \pmod{\mathbb Z}$ for any two indices $i$ and $j$.  
 Let $(\lambda)$ denote the finite dimensional representation of  $\Spin(2n+1)$ with that highest weight. 
Let $\mu=(y_1, \ldots , y_{n-1})\in \Lambda(\Spin(2n-1))$. Then $(\mu)$ appears in the restriction of $(\lambda)$ if and only if 
$x_i\equiv y_j \pmod{\mathbb Z}$ and  $\mu$ 2-step interlaces $\lambda$. 
Let 
\[ 
z_1 \geq z_2 \geq \ldots \geq z_{2n-1}
\] 
be the ordering of $x_i$ and $y_j$.  Recall that the weights for $\Spin(2)$ are half-integers.  For $a\geq 0$, a half integer,  define a  $\Spin(2)$-module 
\[ 
A(a)=(a) + (a-1) + \ldots +(-a), 
\] 
 and for an integer $b\geq 0$, another $\Spin(2)$-module
 \[ 
  B(z) =  (b) + (b-2) + \ldots +(-b).  
  \]  
Then we have \cite{thesis}. 

\begin{prop} \label{P:B_n branching} 
Let $\lambda\in \Lambda(\Spin(2n+1))$ and $\mu\in \Lambda(\Spin(2n-1))$. Assume that $x_i\equiv y_j \pmod{\mathbb Z}$  and 
$\mu$ 2-step interlaces $\lambda$. Then, as $\Spin(2)$-modules, 
\[ 
\Hom_{\Spin(2n-1)} ( (\mu), (\lambda)) \cong B(z_1-z_2) \otimes B(z_3-z_4)\otimes  \cdots \otimes A(z_{2n-1}). 
\] 
\end{prop} 

\subsection{Branching rule from $\Spin(2n)$ to $\Spin(2n-2) \times \Spin(2)$} 
This is similar to the previous case, with the usual annoyances associated to $D_n$-type groups. Let $\lambda\in \Lambda(\Spin(2n))$. 
 Then  $\lambda=(x_1, \ldots , x_n)$ where $x_1, \ldots , x_n$ is a descending sequence of half 
 integers, $x_i \equiv x_j \pmod{\mathbb Z}$, and $x_{n-1} \geq |x_n|$.  
 Let $(\lambda)$ denote the finite dimensional representation of  $\Spin(2n)$ with that highest weight. The group $\Spin(2n)$ has an outer automorphism 
 that changes the sign of $x_n$. Thus,  without loss of generality, we assume that 
 $x_n\geq 0$. Let $\mu=(y_1, \ldots , y_{n-1}) \in \Lambda\Spin(2n-2))$. Let $|\mu|=(y_1, \ldots, |y_{n-1}|)$.  
 Then $(\mu)$ appears in the restriction of $(\lambda)$ if and only if 
$x_i\equiv  y_j\pmod{\mathbb Z}$ and  $|\mu|$ 2-step interlaces $\lambda$. 
Let 
\[ 
z_1 \geq z_2 \geq \ldots \geq z_{2n-1}
\] 
be the ordering of $x_i$ and $|y_j|$.

\begin{prop} \label{P:D_n branching} 
Let $\lambda \in \Lambda(\Spin(2n))$  and $\mu\in \Lambda(\Spin(2n-2))$.  Assume that $x_n\geq 0$, 
 $x_i \equiv y_j \pmod{\mathbb Z}$  and $|\mu|$ 2-step interlaces $\lambda$. Then, as $\Spin(2)$-modules, 
\[ 
\Hom_{\Spin(2n-2)} ( (\mu), (\lambda)) \cong B(z_1-z_2) \otimes B(z_3-z_4)\otimes  \cdots \otimes A(z_{2n-1}). 
\] 
\end{prop}

\subsection{A branching rule from $\rF_4$ to $\Spin(9)$}
We consider the following see-saw pair in~$\rE_{8,4}$.  
\begin{equation} \label{eqseesaw3}
\arraycolsep=1.5pt\def\arraystretch{1.35}
\begin{array}{ccc}
\rF_4 & & \Spin(4,3)\\
| & \bigtimes & | \\
\Spin(9) & & \rG_{2,2}
\end{array}
\end{equation}
where $\rF_4$ denote the compact Lie group of type $\rF_4$.
For $a, b$ non-negative integers where $a \geq b$, we set
\[
\omega(a,b) = (a-b) \varpi_4 + b \varpi_3 = \left(a + \frac{b}{2}, \frac{b}{2}, \frac{b}{2}, \frac{b}{2} \right) \ \in \ \Lambda(\rF_4).
\]
 Let $(\omega(a,b))$ denote the irreducible representation of $\Spin(9)$ with the highest weight $\omega(a,b)$. 

\begin{prop}
Let $w = (w_1,w_2,w_3,w_4) \in \Lambda(\Spin(9))$. Let $(w)$ denote the irreducible representation of $\Spin(9)$ with the highest weight $w$.  
\begin{enumerate}
\item If $w_1 + w_2 > a+b$, then $(\omega(a,b))$ does not contain $(w)$.

\item If $w_1 + w_2 \leq a+b$, then the multiplicity of  $(w)$ in $(\omega(a,b))$ is equal to 
\[
\dim \Hom_{\SU_2}((a-b), (a+b-w_1-w_2)\otimes (w_1-w_2) \otimes (2w_4))
\]
where $(n)$ denotes the irreducible representation of $\SU(2)$ with the highest weight $n$. 
\end{enumerate}
\end{prop}

\begin{proof}
Applying the see-saw pair argument to \eqref{eqseesaw3}, we get
\begin{equation} \label{see_saw_f4} 
 \Hom_{\Spin(9)}((\omega(a,b)), (w)) = 
 \Hom_{\rG_{2,2}}(\Theta((w)), \Theta(\omega(a,b)).  
\end{equation} 
By \cite{HPS},
\[
\Theta(\omega(a,b)) = \bfA(\rG_{2,2}, (a-b)[a+b+10]). 
\]
By \Cref{TthetaSpin9w}
\begin{equation} \label{theta_b_3}  
\Theta((w)) =  \bfA(\Spin(4,3), (w_1-w_2, w_4)[10 + w_1 + w_2]). 
\end{equation}
Moreover, using the filtration in Section \ref{SS:restricting_A}) (which splits because of unitarizability)   (\ref{theta_b_3}) can be rewritten as 
\begin{equation} \label{sum_g_2} 
\Theta((w)) =
\bigoplus_{m=0}^\infty \bfA(\rG_{2,2}, (m) \otimes  (w_1-w_2) \otimes (2w_4)[10 + w_1 + w_2 + m]).
\end{equation} 
Now observe that at most one term in the above sum contributes non-trivially to the right hand side of (\ref{see_saw_f4}), the one such that 
$10 + w_1 + w_2 + m = a + b + 10$, i.e. $m = a + b - w_1 - w_2$. Since $m\geq 0$, the first bullet holds and then, evidently, 
\[ 
 \Hom_{\rG_{2,2}}(\Theta((w)), \Theta(\omega(a,b)) = \Hom_{\SU(2)}((a+b-w_1-w_2)\otimes (w_1-w_2) \otimes (2w_4), (a-b) ). 
\] 
 \end{proof}
The above branching rule generalizes the one of Lepowsky \cite[Section 4]{Lep}, where the case $b=0$ is covered.


\begin{thebibliography}{KMRT} 

\bibitem[AT]{AT} J. Adams and O. Ta\"ibi 
{\em Galois and Cartan cohomology of real groups.} 
Duke Math. J. 167 (2018), no. 6, 1057--1097.



\bibitem[BHLS1]{BHLS} P. Baki\'c, A. Horawa, S. D. Li-Huerta and N. Sweeting {\em Global long root $A$-packets for $G_2$: the dihedral case.} 
arXiv:2405.17375

\bibitem[BHLS2]{BHLS_2} P. Baki\'c, A. Horawa, S. D. Li-Huerta and N. Sweeting {\em Gross's conjecture: the dihedral case.} 
arXiv:2510.03476

\bibitem[Bou]{Bou} Nicolas Bourbaki {\em \'{E}l\'{e}ments de
Mathematique: Groupes et Alg\`{e}bres de Lie.} Chapitres 4, 5 et 6, Hermann (1968). 


\bibitem[GLPS]{GLPS} W. T. Gan,  H. Y. Loke, A. Paul and G. Savin {\em A family of spin(eight) dual pairs: the case of real groups.} 
Symmetry in geometry and analysis. Vol. 1. Festschrift in honor of Toshiyuki Kobayashi, 269--301,
Progr. Math., 357, Birkh\"auser/Springer, Singapore, 2025. 

\bibitem[GT]{GT} I. M. Gelfand and M. L. Tsetlin {\em Finite-dimensional representations of the group of unimodular matrices.}
Dokl. Akad. Nauk SSSR, 71(5), p825 - 828, (1950).





\bibitem[GW1]{GW1} B. Gross and N. Wallach, {\it A distinguished family of
unitary representations for the exceptional groups of real rank = 4.} Lie
Theory and Geometry: In Honor of Bertram Kostant, Progress in Mathematics,
Volume 123, Birkhauser, Boston.

\bibitem[GW2]{GW2} B. Gross and N. Wallach, {\it On quaternionic discrete series representations and their continuations.} 
J reine angew. Math. {\bf 481} (1996),  73--123.


\bibitem[HTW]{HTW} R. Howe, Eng-Chye Tan and J. Willenbring {\em Stable branching rules for classical symmetric pairs.} 
Trans. Amer. Math. Soc. {\bf 357}, no 4, (2005), 1601--1626.  

\bibitem[HPS]{HPS} J. S. Huang, P. Pandzic and G. Savin {\em New dual pair correspondences.} 
Duke Math.  J. {\bf 82}, no 2, (1996),  447--471.

%
%
%

\bibitem[KMRT]{KMRT} M.-A. Knus, A. Merkurjev, M. Rost, and J.-P. Tignol {\em The book of involutions. } 
 American Mathematical Society Colloquium Publications 44, Amer. Math. Soc., Providence, RI, 1998. 

\bibitem[Lep]{Lep} J. Lepowsky {\em Multiplicity formulas for certain semisimple Lie groups.} Bulletin of the AMS {\bf 77}, No. 4, p 601-605 (1977). 


%
%
\bibitem[Li]{LiDuke} Jian-Shu Li {\em The correspondences of infinitesimal
characters for reductive dual pairs in simple Lie groups.} Duke
Math. J. 97, no. 2, (1999) p 347--377.

\bibitem[L1]{thesis} H. Y. Loke {\em Exceptional Lie Algebras and Lie
Groups.} Part 2, Harvard Thesis (1997).

\bibitem[L2]{p33} H. Y. Loke {\em Dual pairs correspondences of $E_{8,4}$ and $E_{7,4}$.}  
Israel Journal of Math. {\bf 113} (1999), 125--162.

\bibitem[L3]{restriction} H. Y. Loke {\em Restrictions of quaternionic representations.} 
Journal of Functional Analysis {\bf 172}  (2000),  377--403.

\bibitem[LS1]{LS1} H. Y. Loke and G. Savin {\em The smallest representations of nonlinear covers of odd orthogonal groups.}
 Amer. J. Math. 130 (2008), no. 3, 763--797.  
 
\bibitem[LS2]{LS2} H. Y. Loke and G. Savin {\em  Dual pair correspondences for non-linear covers of orthogonal groups.}
 J. Funct. Anal. 255 (2008), no. 1, 184--199. 


\bibitem[KP]{KP} W. G. McKay and J. Patera {\em Tables of Dimensions,
Indices, and Branching Rules for Representations of Simple Lie Algebras}, 
Lecture Notes in Pure and Applied Mathematics, Volume 69, M. Dekker (1981).

\bibitem[SR]{SR} S. Salamanca-Riba {\em On the unitary dual of real reductive Lie groups and the $A_{\mathfrak q}(\lambda)$ modules: the strongly regular case.} 
Duke Math. J. {\bf 96}, no. 3, 521--546.   



\bibitem[V]{VoganGreenBook} D. A. Vogan, {\em Representations of real reductive Lie groups}. Progress in Mathematics, vol. 15, Birkh\"{a}user,  (1981).

\bibitem[V2]{G2unitarydual} D. A. Vogan, {\em The unitary dual of $G_2$.} Invent. math. {\bf 116} (1994)  677--791.

\bibitem[VZ]{VoganZuckerman} D. A. Vogan and G. J. Zuckerman, {\em Unitary representations with non-zero cohomology.}
Compositio Mathematica {\bf 53}, no. 1, (1984), p. 51-90.

\bibitem[Wa1]{Transfer} N. Wallach, {\em Transfer of unitary representations between real forms}. Contemporary Math. {\bf 177},  (1994), 181--216.

\bibitem[Wa2]{RRG1} N. Wallach, {\em Real reductive groups I}. Academic Press. Boston (1988).

\bibitem[WY]{WY} N. Wallach and O. Yacobi, {\em A multiplicity formula for tensor products of $\SL_2$-modules and an explicit $\rSp_{2n}$ to $\rSp_{2n-2} \times \rSp_2$ branching formula.} Contemporary Math., {\bf 490}, 151-155 (2009).



\end{thebibliography}
\end{document}